\documentclass[12pt, reqno]{amsart}
\usepackage[margin=1in]{geometry}
\usepackage{amssymb, amsmath, amsthm, amsfonts}
\usepackage[all]{xy}

\usepackage{tikz}
\usetikzlibrary{positioning}

\newtheorem{theorem}{Theorem}[section]
\newtheorem{lemma}[theorem]{Lemma}
\newtheorem{proposition}[theorem]{Proposition}
\newtheorem{corollary}[theorem]{Corollary}
\newtheorem{application}[theorem]{Application}

\newtheorem{predefinition}[theorem]{Definition}
\newenvironment{definition}{\begin{predefinition}\rm}{\end{predefinition}}
\newtheorem{preremark}[theorem]{Remark}
\newenvironment{remark}{\begin{preremark}\rm}{\end{preremark}}
\newtheorem{prenotation}[theorem]{Notation}

\newtheorem{preexample}[theorem]{Example}
\newenvironment{example}{\begin{preexample}\rm}{\end{preexample}}
\newtheorem{preclaim}[theorem]{Claim}

\newtheorem{prequestion}[theorem]{Question}
\newenvironment{question}{\begin{prequestion}\rm}{\end{prequestion}}
 \makeatletter

\title{Ordinary and almost ordinary Prym varieties}
\date{}
\author{Ekin Ozman}
\address{Bogazici University, Faculty of Arts and Sciences, Bebek, Istanbul, 34342, Turkey}
\email{ekin.ozman@boun.edu.tr}

\author{Rachel Pries}
\address{Department of Mathematics, 
Colorado State University, 
Fort Collins, CO 80523, USA}
\email{pries@math.colostate.edu}

\thanks{The first author was partially supported by AWM-NSF Mentoring Travel Grant 2013, TUBITAK 2232 fellowship 114C126 
and Bogazici University Research Grant 15B06SUP3.  
The second author was partially supported by grants NSF DMS-15-02227 and NSA 131011. 
We would like to thank
Achter, Bruin, Casalaina-Martin, Farkas, and Grushevsky 
for good conversations, and the referees for insightful comments.}

\begin{document}

\begin{abstract}
We study the $p$-rank stratification of the moduli space of Prym varieties in characteristic $p > 0$.  
For arbitrary primes $p$ and $\ell$ with $\ell \not = p$ and integers $g \geq 3$ and $0 \leq f \leq g$, 
the first theorem generalizes a result of Nakajima by proving that the Prym varieties of all the unramified ${\mathbb Z}/\ell$-covers 
of a generic curve $X$ of genus $g$ and $p$-rank $f$ are ordinary.
Furthermore, when $p \geq 5$ and $\ell = 2$, 
the second theorem implies that there exists a curve of genus $g$ and $p$-rank $f$ 
having an unramified double cover whose Prym has $p$-rank $f'$ for each $\frac{g}{2}-1 \leq f' \leq g-2$; 
(these Pryms are not ordinary).
Using work of Raynaud, we use these two theorems to prove results about the (non)-intersection of the $\ell$-torsion group scheme
with the theta divisor of the Jacobian of a generic curve $X$ of genus $g$ and $p$-rank $f$.

\end{abstract}

\maketitle

\section{Introduction} 

Suppose $X$ is a smooth projective connected curve of genus $g \geq 2$ 
defined over an algebraically closed field $k$ of characteristic $p > 0$.
Suppose $\pi: Y \to X$ is an unramified cyclic cover of degree $\ell$ for some prime $\ell \not = p$.
Then $Y$ has genus $g_Y=\ell(g-1)+1$ by the Riemann-Hurwitz formula.
For each of the $\ell^{2g}-1$ unramified ${\mathbb Z}/\ell$-covers $\pi:Y \to X$, 
the Jacobian $J_Y$ is isogenous to $J_X \oplus P_{\pi}$ for an abelian variety 
$P_{\pi}$ of dimension $(\ell-1)(g-1)$, called the {\it Prym variety} of $\pi$.
In particular, when $\ell=2$ and $\pi: Y \to X$ is an unramified double cover, 
then $Y$ has genus $2g-1$ and $P_\pi$ is a principally polarized abelian variety of dimension $g-1$.

In this paper, we study the relationship between the $p$-ranks of $J_X$ and $P_{\pi}$.  
The $p$-rank $f_A$ of an abelian variety $A/k$ of dimension $g_A$
is the integer $0 \leq f_A \leq g_A$ such that the number of $p$-torsion points in $A(k)$ is $p^{f_A}$.
One says that $A$ is {\it ordinary} if its $p$-rank is as large as possible ($f_A=g_A$) and 
is {\it almost ordinary} if its $p$-rank equals $g_A-1$.

Consider the moduli space ${\mathcal M}_g$ whose points represent smooth curves $X$ of genus $g$ 
and the moduli space ${\mathcal R}_{g,\ell}$ 
whose points represent unramified ${\mathbb Z}/\ell$-covers $\pi:Y \rightarrow X$. 
There is a finite flat morphism of degree $\ell^{2g}-1$, denoted
\[\Pi_\ell: {\mathcal R}_{g, \ell} \to {\mathcal M}_g,\] which takes the point representing a cover $\pi:Y \to X$ 
to the point representing the curve $X$ \cite[Page 6]{donagismith81}.
For $0 \leq f \leq g$,
let ${\mathcal M}_g^f$ denote the $p$-rank $f$ stratum of ${\mathcal M}_g$.
For most $g$ and $f$, it is not known whether ${\mathcal M}_g^f$ is irreducible; 
however, every component of ${\mathcal M}_g^f$ has dimension $2g-3+f$ \cite[Theorem 2.3]{FVdG}.

By a result of Nakajima, the Prym varieties of the unramified ${\mathbb Z}/\ell$-covers of the generic curve 
$X/k$ of genus $g \geq 2$ are ordinary \cite[Theorem 2]{nakajima}.
In other words, the cover represented by the generic point of ${\mathcal R}_{g,\ell}$ has an ordinary Prym.

The first theorem in the paper generalizes Nakajima's result by adding a condition on the $p$-rank of $X$.
Specifically, if $X$ is a generic $k$-curve of genus $g$ and $p$-rank $f$, then Theorem \ref{Tintro1}
implies that the Prym varieties of all of the unramified ${\mathbb Z}/\ell$-covers of $X$ are ordinary.

Raynaud used the theta divisor $\Theta_X$ in $J_X$ to study the $p$-rank of Prym varieties \cite{raynaudsections, raynaudrevetements, raynaud02}.
(See Section \ref{Stheta} for the definition of $\Theta_X$.)
Using Raynaud's work, 
our theorems yield new results about the (non)-existence of points of order $\ell$ contained in $\Theta_X$.

\begin{theorem} \label{Tintro1} 
Let $\ell \not = p$ be prime.
Let $g \geq 2$ and $0 \leq f \leq g$ with $f \not = 0$ if $g=2$.
Let $S$ be an irreducible component of ${\mathcal M}_g^f$.
\begin{enumerate}
\item (See Theorem \ref{TW}) Then $\Pi^{-1}_\ell(S)$ is irreducible (of dimension $2g-3+f$) 
and the cover represented by the generic point of $\Pi^{-1}_\ell(S)$ has an ordinary Prym.  
\item (See Theorem \ref{Ttheta1}) If $X$ is the curve represented by the generic point of $S$, 
then the theta divisor $\Theta_X$ of the Jacobian of $X$ does not contain any point of order $\ell$.
\end{enumerate}
\end{theorem}

The second theorem
demonstrates the existence of unramified double covers $\pi:Y \to X$ such that the Prym $P_\pi$ is almost ordinary
(with $p$-rank $f'=g-2$).

\begin{theorem} \label{Tintro2}
Let $\ell =2$.  Let $g \geq 2$ and $0 \leq f \leq g$ (with $f \geq 2$ when $p=3$).
Let $S$ be an irreducible component of ${\mathcal M}_g^f$.
\begin{enumerate}
\item (See Theorem \ref{T:NE}) 
The locus of points of $\Pi^{-1}_2(S)$ representing covers whose Prym $P_\pi$ is almost ordinary 
is non-empty with codimension $1$ in $\Pi_2^{-1}(S)$. 
\item (See Theorem \ref{Ttheta2}) The locus of points of $S$ representing curves $X$ for which $\Theta_X$
contains a point of order $2$ is non-empty with codimension $1$ in $S$.
\end{enumerate}
\end{theorem}

As an application of Theorem \ref{Tintro2}, we prove:

\begin{application} \label{A3} (See Corollary \ref{Cdownmore})
Let $\ell=2$ and $p \geq 5$.  Let $g \geq 2$ and $0 \leq f \leq g$.
Then there exists a smooth curve $X/\bar{{\mathbb F}}_p$ of genus $g$ and $p$-rank $f$
having an unramified double cover $\pi:Y \to X$ for which the Prym has $p$-rank $f'$ for each $\frac{g}{2}-1 \leq f' \leq g-1$.
\end{application}

Here is an outline of the paper.
Section \ref{Sppt} contains background about Prym varieties and the $p$-rank stratification of ${\mathcal M}_g$.  
In Section \ref{boundary}, 
we analyze the $p$-ranks of Pryms of covers of singular curves 
and the $p$-rank stratification of the boundary of ${\mathcal R}_{g, \ell}$.

Section \ref{Sordinary} contains the proof of Theorem \ref{Tintro1}(1). 
The proof mirrors Nakajima's technique of degeneration to the boundary of ${\mathcal R}_{g, \ell}$;
the argument is more complicated, however, because ${\mathcal M}_g^f$ may not be irreducible.
To avoid this difficulty, for each irreducible component
$S$ of ${\mathcal M}_g^f$, we consider the ${\mathbb Z}/\ell$-monodromy of the tautological curve $X \to S$, 
namely the image of the fundamental group $\pi_1(S, s_0)$ in ${\rm Aut}({\rm Pic}^0(X)[\ell]_{s_0})$.
A key point is that the ${\mathbb Z}/\ell$-monodromy of $X \to S$ is as large as possible, 
namely ${\rm Sp}_{2g}({\mathbb Z}/\ell)$ \cite[Theorem 4.5]{AP:08}.
We use this to prove that $\Pi^{-1}_\ell(S)$ is irreducible and 
that it degenerates to a particular boundary component $\Delta_{i:g-i}$.

In Sections \ref{Spurity}-\ref{Snotord}, we restrict to the case $\ell=2$.
In Section \ref{Spurity}, we stratify ${\mathcal R}_{g,2}$ by $(f,f')$ 
where $f$ (resp.\ $f'$) is the $p$-rank of $X$ (resp.\ $P_\pi$).
Using purity, we prove that the dimension of each component of the $(f,f')$ stratum of ${\mathcal R}_{g,2}$ is at least $g-2+f+f'$ Proposition~\ref{Phalfway}.
Section \ref{Slowgenus} contains results
about non-ordinary Pryms in the low genus cases $g=2,3$, generalizing \cite[Theorem 6.1]{FVdG}.
  
Section \ref{Snotord} contains the proof of Theorem \ref{Tintro2}(1).  
The proof is again an inductive argument which uses 
the boundary component $\Delta_{i:g-i}$, but it relies on more refined information from 
each of Sections \ref{boundary}, \ref{Sordinary}, \ref{Spurity}, and \ref{Slowgenus}.

Application \ref{A3} (Corollary \ref{Cdownmore}) follows from Theorems \ref{Tintro1} and \ref{Tintro2} 
using a `straight-forward' deformation argument for $f' \leq g-3$:
Suppose $\pi_s$ is an unramified double cover of a {\it singular} curve of genus $g$ and $p$-rank $f$ 
for which the Prym has $p$-rank $f'$.
Applying \cite[Section 3]{AP:08}, one can deform $\pi_s$ to an
unramified double cover of a {\it smooth} curve of genus $g$ whose $p$-rank is still $f$.  
However, it is possible that the $p$-rank $f'$ of the Prym increases in this deformation.
In fact, there are situations where this is guaranteed to happen, see Remark \ref{Rsingular}.
Under the hypotheses of Corollary \ref{Cdownmore}, 
we construct a deformation of $\pi_s$ for which the $p$-rank of the Prym remains constant.
We emphasize that the technique used in Theorems~\ref{Tintro1} - \ref{Tintro2}
is stronger than the straight-forward approach
and gives more information about the $p$-rank stratification of ${\mathcal R}_{g,\ell}$. 

Section \ref{Stheta} contains the definition of the theta divisor $\Theta_X$ and the proofs of 
Theorems \ref{Tintro1}(2) and \ref{Tintro2}(2).
We then compare our results with those of Raynaud \cite{raynaudsections, raynaudrevetements} and Pop/Saidi \cite{popsaidi}.
Briefly, Raynaud's results are stronger in that they apply to an arbitrary base curve $X$ 
but are weaker in other ways: his result about ordinary Pryms applies only when $\ell >(p-1)3^{g-1}g!$, 
in which case he shows that at least one of the Pryms is ordinary;
and, in his result for non-ordinary Pryms, the Galois group of the cover
is solvable but not cyclic and the $p$-rank of the Prym is not determined.
In \cite[Proposition 2.3]{popsaidi}, the result is stronger in that 
it applies to an arbitrary curve which is either non-ordinary or whose Jacobian is absolutely simple,
but is weaker in that the degree of the cyclic unramified cover and the $p$-rank of the Prym are not determined.

Section \ref{Sopen} contains some open questions.
For example, in Section~\ref{Sexample}, we illustrate the difficulty in proving these results computationally, 
even for $g=2$ and a fixed small prime $p$.
See \cite{Bouw} for results about $p$-ranks of {\it ramified} cyclic covers of curves.

\section{Prym varieties and $p$-rank stratifications} \label{Sppt}

Suppose $X$ is a smooth projective curve of genus $g$ defined over $k$.
The Jacobian $J_X$ of $X$ is a principally polarized abelian variety of dimension $g$.
A ${\mathbb Z}/\ell$-cover $\pi: Y \to X$ is a Galois cover together with an isomorphism
$\iota: {\rm Gal}(Y/X) \to {\mathbb Z}/\ell$.
For a prime $\ell \not = p$, 
there is a bijection between points of order $\ell$ on $J_X$ and unramified connected 
${\mathbb Z}/\ell$-covers $\pi:Y \to X$.

\subsection{Prym varieties}

Suppose $\pi: Y \rightarrow X$ is an unramified ${\mathbb Z}/\ell$-cover. 
The {\it Prym variety} $P_\pi$ is the connected component containing $0$ of 
the norm map on Jacobians. 
More precisely, if $\sigma$ is the endomorphism of $J_Y$ induced by a generator $\sigma$ of ${\rm Gal}(Y/X)$, then 
 \[P_{\pi}={\rm Im}(1-\sigma)=\ker(1+\sigma+\ldots+\sigma^{\ell-1})^0.\] 
The canonical principal polarization of $J_Y$ induces a polarization on $P_\pi$ \cite[Page 6]{LO}. 
This polarization is principal when $\ell=2$.

\subsection{Moduli spaces of unramified cyclic covers}

Let $\mathcal R_{g,\ell}$ denote the moduli space whose points represent unramified
${\mathbb Z}/\ell$-covers of smooth projective curves of genus $g$; it is a smooth Deligne-Mumford stack \cite[Page 5]{CF}.

The points of ${\mathcal R}_{g,\ell}$ can also represent triples $(X,\eta,\phi)$ 
where $X$ is a smooth genus $g$ curve equipped with a line bundle $\eta \in {\rm Pic}(X)$ 
and an isomorphism $\phi:\eta^{\otimes \ell} \stackrel{\sim}{\to} \mathcal O_X$.  
This is because the data of the ${\mathbb Z}/\ell$-cover $\pi:Y \to X$ is equivalent to the data $(X, \eta, \phi)$.

The morphism $\Pi_\ell: {\mathcal R}_{g, \ell} \to {\mathcal M}_g$, which sends the point representing $\pi:Y \to X$ to the point 
representing $X$, is surjective, \'etale, and finite of degree $\ell^{2g}-1$.  Thus $\dim({\mathcal R}_{g, \ell})=3g-3$.

\subsection{Marked covers} \label{Smarkedcurves}

A point of ${\mathcal M}_{g;1}$ represents a smooth curve $X$ of genus $g$ together with a marking, namely 
the choice of a point $x \in X$.
A point of ${\mathcal R}_{g, \ell;1}$ represents an unramified ${\mathbb Z}/\ell$-cover $\pi: Y \to X$, 
where $X$ is a smooth curve of genus $g$, together with a marking $x' \mapsto x$, namely the choice of a point $x' \in \pi^{-1}(x)$. 
The marking $x' \mapsto x$ determines
a labeling of the $\ell$ points of the fiber $\pi^{-1}(x)$ because of the ${\mathbb Z}/\ell$-action.
 There are forgetful maps $\psi_M: {\mathcal M}_{g;1} \to {\mathcal M}_g$ 
 and $\psi_R: {\mathcal R}_{g, \ell;1} \to {\mathcal R}_{g, \ell}$.

\begin{lemma} \label{Lirreducible}
If $S \subset {\mathcal M}_g$ is irreducible, then $\psi_M^{-1}(S)$ is irreducible in ${\mathcal M}_{g;1}$.
If $Q \subset {\mathcal R}_{g,\ell}$ is irreducible, then $\psi_R^{-1}(Q)$ is irreducible in ${\mathcal R}_{g, \ell ;1}$.
\end{lemma}

\begin{proof}
The fiber of $\psi_M$ above the point of ${\mathcal M}_g$ representing $X$ is isomorphic to $X$ and is thus irreducible.
The fiber of $\psi_R$ above the point of ${\mathcal R}_{g, \ell}$ representing $\pi:Y \to X$ is isomorphic to $Y$ and is thus irreducible.  
The result follows from Zariski's theorem.
\end{proof}

\subsection{The $p$-rank}

Let $\mu_p$ be the kernel of Frobenius on ${\mathbb G}_m$.
The $p$-rank of a semi-abelian variety $A'$ is $f_{A'}={\rm dim}_{{\mathbb F}_p}{\rm Hom}(\mu_p, A')$.
If $A'$ is an extension of an abelian variety $A$ by a torus $T$, then $f_{A'}=f_A + {\rm rank}(T)$.

For an abelian variety $A$,     
the $p$-rank can also be defined as the integer $f_A$ such that 
the number of $p$-torsion points in $A(k)$ is $p^{f_A}$.
If $A$ has dimension $g_A$ then $0 \leq f_A \leq g_A$.
The $p$-rank is invariant under isogeny $\sim$ of abelian varieties.

The $p$-rank of a stable curve $X$ is that of ${\rm Pic}^0(X)$.
There are three $p$-ranks associated with an unramified cover $\pi:Y \to X$, namely the $p$-rank $f$ of $X$, 
the $p$-rank $f'$ of $P_{\pi}$, and the $p$-rank of $Y$ which equals $f+f'$.

\subsection{The $p$-rank stratification}

If $X/S$ is a
semi-abelian scheme over a Deligne-Mumford stack, there is a
stratification $S = \cup S^{f}$ by locally closed reduced substacks such that $s \in S^f(k)$ if and only if 
$f(X_s) =f$ (\cite[Theorem 2.3.1]{katzsf}, see also \cite[Lemma
2.1]{AP:08}).  So, ${\mathcal M}_g^f$ is the locally closed reduced substack of ${\mathcal M}_g$
whose points represent smooth curves of genus $g$ with $p$-rank $f$.

\subsection{Compactification of ${\mathcal M}_g$}

Suppose $X$ is a stable curve with irreducible components $C_i$, for $1 \leq i \leq s$.
Let $\tilde{C}_i$ be the normalization of $C_i$.
By \cite[Example 8, Page 246]{BLR}, $J_X$ is canonically an extension of an abelian variety by a torus $T$.
There is a
short exact sequence:
\begin{equation} \label{eqBLR}
1 \rightarrow T \rightarrow J_X \rightarrow \oplus_{i=1}^s J_{\tilde{C}_i} \rightarrow 1.
\end{equation}
The rank $r_T$ of $T$ is the rank of the cohomology group $H^1(\Gamma_X,\mathbb Z)$,  
where $\Gamma_X$ denotes the dual graph of $X$. 
One says that $X$ has {\it compact type} if $T$ is trivial.

Let $\bar {\mathcal M}_g$ be the Deligne-Mumford compactification of ${\mathcal M}_g$;
it is a smooth proper Deligne-Mumford stack. 
The boundary $\partial {\mathcal M}_g =\bar {\mathcal M}_g - {\mathcal M}_g$ is the union of the components $\Delta_0[\bar {\mathcal M}_g]$ and 
$\Delta_i[\bar {\mathcal M}_g]$ for $1 \leq i \leq \lfloor g/2 \rfloor$ defined as in \cite[Section 2.3]{AP:08}.

For $1 \leq i \leq \lfloor g/2 \rfloor$, 
the $p$-rank $f$ stratum of the boundary component $\Delta_i[\bar {\mathcal M}_g]$ is the union of the images of the clutching morphisms:
\begin{equation} \label{jacisoF}
\kappa_i: \bar {\mathcal M}_{i;1}^{f_1} \times \bar {\mathcal M}_{g-i; 1}^{f_2} \to \Delta_i[\bar {\mathcal M}_g^{f_1+f_2}],
\end{equation}
for all pairs $\{f_1, f_2\}$ of non-negative integers such that $f_1+f_2=f$.
The $p$-rank $f$ stratum of the boundary component $\Delta_0[\bar {\mathcal M}_g]$ is the image of the clutching morphism:
\begin{equation} \label{noncompactisoF}
\kappa_0: \bar {\mathcal M}_{g-1;2}^{f-1} \to \Delta_0[\bar {\mathcal M}_g^{f}].
\end{equation}

\section{The $p$-rank stratification of the boundary of ${\mathcal R}_{g,\ell}$} \label{boundary} 

In this section, we study Pryms of unramified covers of singular curves.
Then we analyze the $p$-rank stratification of the boundary of ${\mathcal R}_{g,\ell}$, whose points typically represent 
unramified ${\mathbb Z}/\ell$-covers of singular curves.
Lemma \ref{LdimW}
states that every component of the boundary of $\Pi^{-1}_\ell({\mathcal M}_{g}^f)$ has dimension $2g-4+f$.
This is used for Propositions \ref{TdegenW} and \ref{P:NE}.
We compute the $p$-rank of the Prym of the cover represented by the generic point
of each boundary strata;
in particular, Lemma \ref{LDelig-i} is used in Sections \ref{Sordinary}, \ref{Slowgenus}, and \ref{Snotord}.

This section relies on structural results from \cite{F_generalL} and \cite{CF}.
It is necessary to include some material from these references.
The following lemma will also be useful.

\begin{lemma} \label{Lintersect}  \cite[page 614]{V:stack}.
If $A$ and $B$ are substacks of a smooth proper stack $S$ then
\[{\rm codim}(A \cap B, S) \leq {\rm codim}(A, S) + {\rm codim}(B,S).\]
\end{lemma}

\subsection{Compactification of ${\mathcal R}_{g, \ell}$}

By \cite[Definition 1.2]{CF}, a twisted curve $\texttt{C}$ is a one dimensional stack such that the corresponding coarse moduli space $C$ is a stable curve whose smooth locus is represented by a scheme and whose singularities are nodes with local picture $[\{xy=0\}/\mu_r]$ with $\zeta \in \mu_r$ acting as $\zeta (x,y)=(\zeta x, \zeta^{-1}y)$.
The definition of a faithful line bundle $\eta \in {\rm Pic}(\texttt C)$ is in \cite[Definition 1.3]{CF}. 
By \cite[Definition 1.5]{CF}, a level-$\ell$ twisted curve of genus $g$ is a triple $[\texttt X,\eta,\phi]$  
where $\texttt{X}$ is a twisted curve of genus $g$,
$\eta \in {\rm Pic}(\texttt X)$ is a faithful line bundle, 
and $\phi: \eta^{\otimes \ell} \rightarrow \mathcal{O}_{\texttt X}$ is an isomorphism. 

By \cite[page 6]{CF}, 
the moduli space ${\mathcal R}_{g,\ell}$ admits a compactification $\bar{{\mathcal R}}_{g,\ell}$ whose points represent  
level-$\ell$ twisted curves of genus $g$.   
It is a smooth Deligne-Mumford stack and there is a finite forgetful  
morphism $\Pi_\ell: \bar {\mathcal R}_{g, \ell} \to \bar {\mathcal M}_g$.

Let $\partial {\mathcal R}_{g, \ell} = \bar {\mathcal R}_{g, \ell} - {\mathcal R}_{g, \ell}$.
Some points of $\partial {\mathcal R}_{g, \ell}$ cannot be interpreted in terms of $\ell$-torsion line bundles 
or ${\mathbb Z}/\ell$-covers of a scheme-theoretic curve.  
For the sake of intuition, whenever possible, 
we describe the generic point of a boundary component of $\bar{{\mathcal R}}_{g, \ell}$ 
in terms of the cover $\pi:Y \to X$ it represents.

\subsection{Definition of $W_g^f$}

\begin{definition}
For $0 \leq f \leq g$, 
define $W_g^f=\Pi_\ell^{-1}({\mathcal M}_g^f)$ and $\bar W_g^f=\Pi_\ell^{-1}(\bar {\mathcal M}_g^f)$.  
\end{definition}

The points of $W_g^f$ represent unramified ${\mathbb Z}/\ell$-covers 
$\pi: Y \to X$ of a smooth curve $X$ of genus $g$ and $p$-rank $f$.

\begin{lemma}\label{LdimW}
Let $g \geq 1$ and $0 \leq f \leq g$.  Then $W_g^f$ is non-empty.
For $g \geq 2$, let $Q$ be an irreducible component of $\bar W_g^f$.  Then 
\begin{enumerate}
\item $Q$ has dimension $2g-3+f$;
\item $W_g^f \cap Q$ is open and dense in $Q$ (the generic point of $Q$ represents a smooth curve);
\item the dimension of 
every component of $Q \cap \partial {\mathcal R}_{g, \ell}$ is $2g-4+f$.
\end{enumerate}
\end{lemma}

\begin{proof}
Since $\Pi_\ell: \bar {\mathcal R}_{g, \ell} \rightarrow \bar {\mathcal M}_g$ is finite, flat, and surjective, these facts follow from the 
analogous facts for $\bar {\mathcal M}_g^f$; see \cite[Theorem 2.3]{FVdG} for part (1) and \cite[Lemmas 3.1, 3.2(a)]{AP:08} 
for parts (2) - (3).
\end{proof}

\subsection{Boundary components of $\bar{{\mathcal R}}_{g, \ell}$}

Let $\partial {\mathcal R}_{g, \ell}=\bar{{\mathcal R}}_{g, \ell} - {\mathcal R}_{g, \ell}$ denote the boundary of 
${\mathcal R}_{g, \ell}$.
Informally, the points of $\partial {\mathcal R}_{g, \ell}$ represent
unramified ${\mathbb Z}/\ell$-covers of singular curves, although we make this more precise below.
For covers of singular curves of compact type,
the boundary components lie above $\Delta_i[\bar {\mathcal M}_g]$ for some $1 \leq i \leq \lfloor g/2 \rfloor$ and
are denoted $\Delta_{i:g-i}$, $\Delta_i$, and $\Delta_{g-i}$.
For covers of singular curves of non-compact type, 
the boundary components lie above $\Delta_0[\bar {\mathcal M}_g]$ and
are denoted $\Delta_{0,I}$, $\Delta_{0,II}$ and $\Delta_{0,III}^{(a)}$.

In Sections \ref{compacttype} and \ref{Snoncompact}, we recall the definition of these boundary components
and investigate them in terms of the $p$-rank.  
Before doing this, recall the following results.

\begin{proposition} \label{Pdivisor1}  \cite[Equation (16)]{F_generalL}
For $1 \leq i < \lfloor g/2 \rfloor$, there is an equality of divisors 
\[\Pi_\ell^*(\Delta_i[\bar {\mathcal M}_g])=\Delta_i[\bar {\mathcal R}_{g,\ell}] + \Delta_{g-i}[\bar {\mathcal R}_{g,\ell}]+\Delta_{i:g-i}[\bar {\mathcal R}_{g,\ell}].\]
If $g$ is even, there is an equality of divisors 
\[\Pi_\ell^*(\Delta_{g/2}[\bar {\mathcal M}_g])=\Delta_{g/2}[\bar {\mathcal R}_{g,\ell}] +\Delta_{g/2:g/2}[\bar {\mathcal R}_{g,\ell}].\]
\end{proposition}

\begin{proposition} \label{Pdivisor2} \cite[page 89]{F_generalL} or \cite[Equation (17)]{CF}
There is an equality of divisors 
\[\Pi_{\ell}^*(\Delta_0[\bar {\mathcal M}_g])=\Delta_{0,I}[\bar {\mathcal R}_{g,\ell}] + \Delta_{0,II}[\bar {\mathcal R}_{g,\ell}] + 
\ell \sum_{a=0}^{\lfloor \ell/2 \rfloor} \Delta_{0,III}^{(a)}[\bar {\mathcal R}_{g,\ell}].\]
\end{proposition}

\subsection{Pryms of covers of singular curves of compact type} \label{compacttype}

Let $X$ be a singular curve formed by intersecting two curves $C_1$ and $C_2$ 
(at points $x_1 \in C_1$ and $x_2 \in C_2$) in an ordinary double point.
By \eqref{eqBLR}, $J_X \backsimeq J_{C_1} \oplus J_{C_2}$.
Let $\xi$ be the point of $\Delta_i[\bar {\mathcal M}_g]$ representing $X$.
Then an unramified cyclic degree $\ell$ cover $\pi:Y \to X$ is determined by two line bundles 
$\eta_{C_1} \in {\rm Pic}^0(C_1)[\ell]$ and $\eta_{C_2} \in {\rm Pic}^0(C_2)[\ell]$, which are not both trivial. 
The points of $\Delta_{i: g-i}[\bar {\mathcal R}_{g,\ell}]$ above $\xi$ represent covers $\pi$
for which both $\eta_{C_1}$ and $\eta_{C_2}$ are nontrivial;
the points of $\Delta_i[\bar {\mathcal R}_{g,\ell}]$ (resp.\ $\Delta_{g-i}[\bar {\mathcal R}_{g,\ell}]$) above 
$\xi$ represent covers $\pi$ for which $\eta_{C_1}$ (resp.\ $\eta_{C_2}$) is trivial.

\subsubsection{The boundary component $\Delta_{i:g-i}$} \label{SDeltaig-i}

The boundary divisor $\Delta_{i:g-i}[\bar {\mathcal R}_{g, \ell}]$ is the image of the clutching map
\[\kappa_{i:g-i}: \bar {\mathcal R}_{i,\ell;1} \times \bar {\mathcal R}_{g-i,\ell; 1} \rightarrow \bar {\mathcal R}_{g,\ell},\]
defined on a generic point as follows.  
Let $\tau_1$ be a point of $\bar {\mathcal R}_{i,\ell;1}$ 
representing $(\pi_1:C'_1 \to C_1, x'_1 \mapsto x_1)$
and let $\tau_2$ be a point of $\bar {\mathcal R}_{g-i,\ell;1}$ representing $(\pi_2:C'_2 \to C_2, x'_2 \mapsto x_2)$.
Let $Y$ be the curve with components $C'_1$ and $C'_2$, formed by identifying
$\sigma^k(x'_1)$ and $\sigma^k(x'_2)$ in an ordinary double point for $0 \leq k \leq \ell-1$.
Then $\kappa_{i:g-i}(\tau_1, \tau_2)$ is the point representing the unramified ${\mathbb Z}/\ell$-cover $Y \to X$. 
This is illustrated in Figure \ref{fig:Deltagi} for $\ell=2$. 

\begin{figure}[h]
\centering
\includegraphics[scale=0.5]{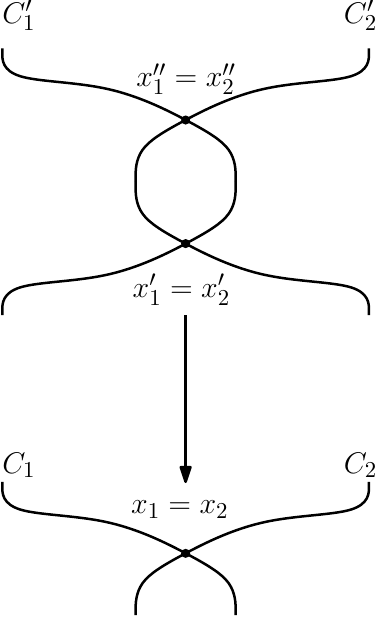}
\caption{$\Delta_{i:g-i}: \eta_{C_1} \ncong \mathcal O_{C_1}, \ \eta_{C_2} \ncong \mathcal O_{C_2}$}
\label{fig:Deltagi}
\end{figure}

\begin{lemma}\label{dimig-i} The clutching map $\kappa_{i:g-i}$ restricts to a map
\[\kappa_{i:g-i}: \bar W_{i;1}^{f_1} \times \bar W_{g-i;1}^{f_2} \rightarrow \Delta_{i:g-i}[\bar W_g^{f_1+f_2}].\]
\end{lemma}

\begin{proof} This follows from \eqref{jacisoF}.
\end{proof}

\begin{lemma} \label{LDelig-i}
Suppose $\pi:Y \rightarrow X$ is an unramified ${\mathbb Z}/\ell$-cover represented by a point of 
$\Delta_{i:g-i}[\bar {\mathcal R}_{g,\ell}]$.  
Then $P_\pi$ is an extension of a semi-abelian variety $P_{\pi_1} \oplus P_{\pi_2}$ by a torus $T$ of 
rank $r_T=\ell - 1$. 
If $f_i'$ is the $p$-rank of $P_{\pi_i}$, then the $p$-rank of $P_\pi$ is $f_1'+f_2'+(\ell-1)$.
\end{lemma}

\begin{proof}
By \eqref{eqBLR}, $J_Y$ is an extension of $J_{C'_1} \oplus J_{C'_2}$ by a torus $T$
whose rank $r_T$ is the rank of $H^1(\Gamma_Y,\mathbb Z)$. 
Then $r_T=\ell-1$ since $\Gamma_Y$ consists of two vertices, 
for the two irreducible components $C_1', C_2'$, 
which are connected with $\ell$ edges, 
corresponding to the $\ell$ intersection points.
There is a commutative diagram with exact rows
\[\xymatrix{
      1\ar[r]& 1 \ar[r]\ar[d]^{a}& J_X \ar[r]\ar[d]^{b}&J_{C_1} \oplus J_{C_2} \ar[r]\ar[d]^{c}&1\\
      1\ar[r]&T\ar[r]& J_Y \ar[r]& J_{C_1'} \oplus J_{C_2'} \ar[r]&1.}\]
By the snake lemma, there is an exact sequence
\[1 \to {\rm Coker}(a)  \to {\rm Coker}(b)  \to {\rm Coker}(c) \to 1,\] 
and thus an exact sequence
\begin{equation} \label{Esnake}
1 \to T \to P_\pi \to P_{\pi_1} \oplus P_{\pi_2} \to 1.
\end{equation}
From \eqref{Esnake},
$f_{P_\pi}=f_{P_{\pi_1} \oplus P_{\pi_2}} + {\rm rank}(T) = f'_1+f'_2+(\ell-1)$.
\end{proof}

\subsubsection{The boundary component $\Delta_i, i>0$:} \label{Sdeltai}
The boundary divisor $\Delta_i[\bar {\mathcal R}_g]$ is the image of the clutching map
\[\kappa_i: \bar {\mathcal R}_{i,\ell;1} \times \bar {\mathcal M}_{g-i; 1} \rightarrow \bar {\mathcal R}_{g,\ell},\]
defined on a generic point as follows.  
Let $\tau$ be a point of ${\mathcal R}_{i,\ell;1}$ representing $(\pi_1':C'_1 \to C_1,x' \mapsto x)$
and let $\omega$ be a point of ${\mathcal M}_{g-i; 1}$ representing $(C_2,x_2)$.
Let $Y$ be the curve with components $C'_1$ and $\ell$ copies of $C_2$, formed by identifying the point $x_2$ 
on each copy of $C_2$ with a point of $\pi_1^{-1}(x_1)$.
Then $\kappa_i(\tau, \omega)$ represents the unramified ${\mathbb Z}/\ell$-cover $\pi:Y \to X$.   
See Figure \ref{fig:Deltai} for $\ell=2$.

\begin{figure}[h]
\centering
\includegraphics[scale=0.5]{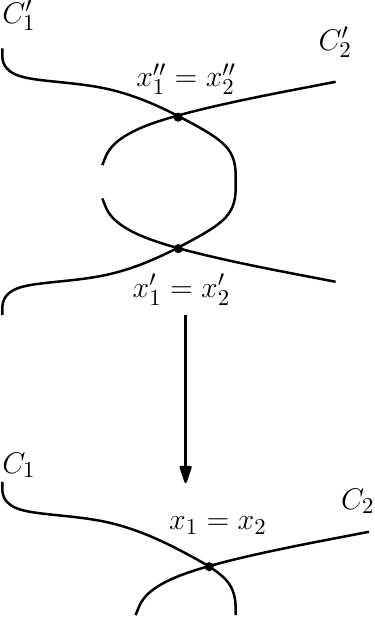}
\caption{$\Delta_{i}:\eta_{C_1} \not \simeq {\mathcal O}_{C_1}, \eta_{C_2} \simeq {\mathcal O}_{C_2}$}
\label{fig:Deltai}
\end{figure}

\begin{lemma}\label{dimi}
The clutching map $\kappa_i$ restricts to a map:
 \[\kappa_i: \bar W_{i;1}^{f_1} \times \bar {\mathcal M}_{g-i; 1}^{f_2} \rightarrow \Delta_i[\bar W_g^{f_1+f_2}].\]
\end{lemma}

\begin{proof}
This follows from \eqref{jacisoF}. 
\end{proof}

\begin{lemma} \label{LDeli}
Suppose $\pi:Y \rightarrow X$ is an unramified ${\mathbb Z}/\ell$-cover represented by a point of 
$\Delta_{i}[\bar {\mathcal R}_{g,\ell}]$.  
Then $P_\pi \simeq P_{\pi_1} \oplus J_{C_2}^{\ell-1}$. 
If $f'_1$ is the $p$-rank of $P_{\pi_1}
$, then the $p$-rank of $P_\pi$ is $f'_1+ (\ell-1)f_2$.
\end{lemma}

\begin{proof}
By construction, $J_Y \backsimeq J_{C'_1} \oplus J_{C_2}^\ell$. 
The image of $(1-\sigma)$ on $J_Y$ is $P_\pi$, while on 
$J_{C'_1} \oplus J_{C_2}^\ell$ it is $P_{\pi_1} \oplus J_{C_2}^{\ell-1}$. 
Then the $p$-rank is the sum of the $p$-ranks of $P_{\pi_1}$ and $J_{C_2}^{\ell-1}$. 
\end{proof}


\subsection{Pryms of covers of singular curves of non-compact type} \label{Snoncompact}

This material is needed only for future work.
The main reference is \cite[Example 6.5]{Farkas} when $\ell=2$ 
and \cite[Section 1.4]{F_generalL} and \cite[Section 1.5.2]{CF} for general $\ell$.
Let $(X',x,y)$ be a curve of genus $g-1$ with 2 marked points.
Let $X$ be a curve of genus $g$ of non-compact type formed by identifying 
two points $x,y$ on $X'$.
By \eqref{eqBLR}, if $X'$ has $p$-rank $f_1$, then $X$ has $p$-rank $f=f_1+1$.

\subsubsection{The Boundary Component $\Delta_{0,I}$} 
The boundary divisor $\Delta_{0,I}[\bar {\mathcal R}_{g,\ell}]$ is the image of the clutching map
\[\kappa_{0, I}: \bar {\mathcal R}_{g-1,\ell;2} \rightarrow \bar {\mathcal R}_{g,\ell},\]
defined on a generic point as follows.  
Let $\tau$ be a point of ${\mathcal R}_{g-1,\ell;2}$ representing $(\pi':Y' \to X', x'\mapsto x, y'\mapsto y)$ (two markings).
Let $Y$ be the nodal curve of non-compact type with normalization $Y'$, 
formed by identifying $\sigma^k(x')$ and $\sigma^k(y')$, for $0 \leq k \leq \ell-1$, in an ordinary double point.
Then $\kappa_{0, I}(\tau)$ is the point representing the unramified ${\mathbb Z}/\ell$-cover $\pi:Y \to X$.
See Figure \ref{fig:Delta01} for the case $\ell=2$.

\begin{figure}[h]
\centering
\includegraphics[width=0.5\linewidth]{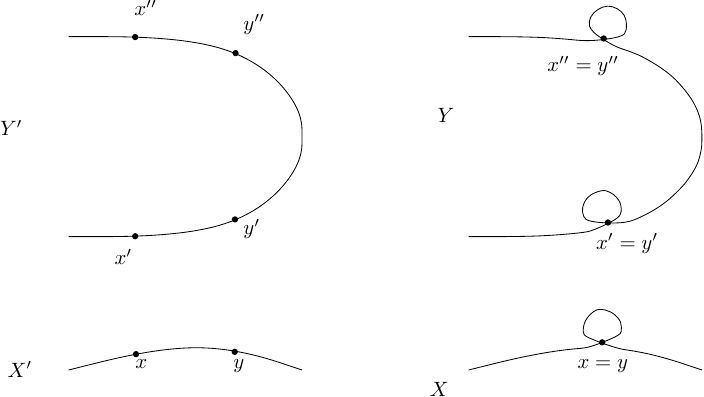}
\caption{$\Delta_{0,I}, \ell=2$}
\label{fig:Delta01}
\end{figure}

\begin{lemma}\label{bounW1}
The clutching map $\kappa_{0,I}$ restricts to a map: $\kappa_{0,I}: \bar W_{g-1;2}^{f_1} \to \Delta_{0,I}[\bar W_g^{f_1+1}]$.
\end{lemma}

\begin{proof}
This follows from \eqref{noncompactisoF}.
\end{proof}

\begin{lemma} \label{Lprank1}
Suppose $\pi:Y \rightarrow X$ is an unramified ${\mathbb Z}/\ell$-cover represented by a point of 
$\Delta_{0,I}[\bar {\mathcal R}_{g,\ell}]$.  
Then $P_\pi$ is an extension of a semi-abelian variety $P_{\pi'}$ by a torus $T$ with $r_T =\ell-1$.
If $f'_1$ is the $p$-rank of $P_{\pi'}$, then the $p$-rank of $P_\pi$ is $f'=f'_1+(\ell-1)$.
\end{lemma}

\begin{proof}
There is a commutative diagram with exact rows:
\[\xymatrix{
      1\ar[r]& T_X \ar[r]\ar[d]& J_X \ar[r]\ar[d]&J_{X'} \ar[r]\ar[d]&1\\
      1\ar[r]&T_Y\ar[r]& J_Y \ar[r]& J_{Y'} \ar[r]&1,}\]
where $T_X$ is a torus of rank $1$ and $T_Y$ is a torus of rank $\ell$.
Then $T_Y/T_X$ is a torus $T$ of rank $\ell-1$.
By the snake lemma, there is an exact sequence
$1 \to T \to P_\pi \to P_{\pi'} \to 1$. 
\end{proof}

\subsubsection{The Boundary Component $\Delta_{0, II}$} The boundary divisor $\Delta_{0,II}[\bar {\mathcal R}_{g,\ell}]$ is the image of the clutching map
\[\kappa_{0,II}: \bar {\mathcal M}_{g-1;2} \rightarrow \bar {\mathcal R}_{g,\ell},\]
defined on a generic point as follows.
Let $\omega$ be a point of $\bar {\mathcal M}_{g-1; 2}$ representing $(X', x, y)$ (with 2 markings). 
Consider a disconnected curve with components $(X'_i, x_i,y_i)$
indexed by $i \in {\mathbb Z}/\ell$ such that each component is isomorphic to $(X', x, y)$. 

Let $Y$ be the nodal curve of non-compact type formed by identifying $\sigma^k(x_1)$ and $\sigma^{k+1}(y_1)$,
for $0 \leq k \leq \ell-1$,  
in an ordinary double point.
Then $\kappa_{0,II} (\omega)$ is the point representing the ${\mathbb Z}/\ell$-cover $\pi:Y \to C$.
This is illustrated in Figure \ref{fig:Delta02} for $\ell=2$.

\begin{figure}[h]
\centering
\includegraphics[width=0.6\linewidth]{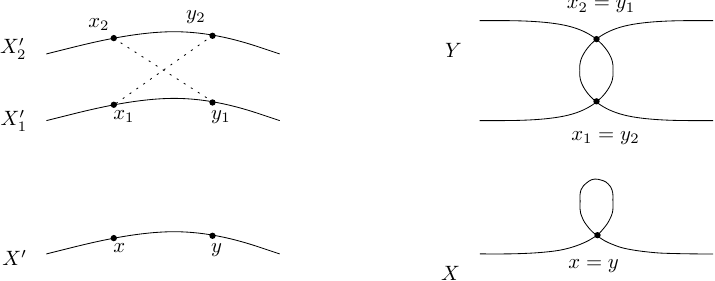}
\caption{$\Delta_{0,II, \ell=2}$}
\label{fig:Delta02}
\end{figure}

\begin{lemma}\label{bounW2}
The clutching map $\kappa_{0,II}$ restricts to a map $\kappa_{0,II}: \bar {\mathcal M}_{g-1;2}^{f_1} \to \Delta_{0,II}[\bar W_g^{f_1+1}]$.
\end{lemma}

\begin{proof}
This follows from \eqref{noncompactisoF}.
\end{proof}

\begin{lemma} \label{Lprank2}
Suppose $\pi:Y \rightarrow X$ is an unramified ${\mathbb Z}/\ell$-cover represented by a point of 
$\Delta_{0,II}[\bar {\mathcal R}_{g,\ell}]$.  
Then $P_\pi \simeq J_{X'}^{\ell-1}$.
If $f_1$ is the $p$-rank of $X'$, then the $p$-rank of $P_\pi$ is $(\ell-1)f_1$.
\end{lemma}

\begin{proof}
Like Lemma \ref{Lprank1}, but $T_X$ and $T_Y$ have rank $1$.
\end{proof}

\subsubsection{The Boundary Component $\Delta_{0,III}$}  

The points of the last boundary component(s) represent
level-$\ell$ twisted curves $[X, \tilde{\eta}, \tilde{\phi}]$ which have the following structure.
Let $(X', x, y)$ be a point of $\bar {\mathcal M}_{g-1; 2}$ and let $E$ be a projective line.
The curve $X=X' \cup_{x,y} E$ of genus $g$ has components $X'$ and $E$, 
with two ordinary double points formed by identifying $x$ with $0$ and $y$ with $\infty$.
Since $E$ is an exceptional component, the restriction of $\tilde{\eta}$ to $E$ is 
$\eta_E=\mathcal O_E(1)$.  
This implies that the restriction $\eta_{X'}$ of $\tilde{\eta}$ to $X'$ has degree $-1$ and
that
$\eta_{X'}^{\otimes - \ell}=\mathcal O_{X'}(a x + (\ell-a) y)$,
for some $1 \leq a \leq \ell-1$.

The boundary divisor $\Delta_{0,III}^{(a)}$ is the closure in $\bar {\mathcal R}_{g-1,\ell}$ of points representing 
such level-$\ell$ twisted curves
$[X, \tilde{\eta}, \tilde{\phi}]$.  For complete details, see \cite[Section 1.4]{F_generalL}.
There is a clutching map 
\[\kappa_{0, III}^{(a)}: \bar M_{g-1;2} \rightarrow \Delta_{0,III}^{(a)}[\bar R_{g, \ell}].\]

\begin{lemma} \label{bounW3}
The clutching map $\kappa_{0,III}^{(a)}$ restricts to a map 
\[\kappa_{0,III}^{(a)}: \bar W_{g-1;2}^{f_1} \rightarrow \Delta_{0,III}^{(a)}[\bar W_g^{f_1+1}].\]
\end{lemma}

\begin{proof}
This follows from \eqref{noncompactisoF}.
\end{proof}

\section{Ordinary Pryms of a generic curve of given $p$-rank} \label{Sordinary}

The main result of this section is that the Prym variety of an unramified ${\mathbb Z}/\ell$-cover of a 
generic curve of genus $g$ and $p$-rank $f$ is ordinary, Theorem \ref{TW}.  We prove this using degeneration to
$\partial {\mathcal R}_{g, \ell}$ and information about the ${\mathbb Z}/\ell$-monodromy of components of ${\mathcal M}_g^f$.  
The monodromy results are needed since all curves represented by a point of $\Delta_i[\bar {\mathcal M}_g^f]$ 
have an unramified ${\mathbb Z}/\ell$-cover with non-ordinary Prym, when $f<g$ and $0 \leq i <g$,
as seen in Sections \ref{Sdeltai} and \ref{Snoncompact}.

\subsection{Earlier work on the $p$-rank stratification of ${\mathcal M}_g$} 

\begin{proposition} \label{PinterDelta1}
\cite[Proposition 3.4]{AP:08}
Let $g \ge 2$ and $0 \le f \le g$.  
Suppose $1 \le i \le g-1$ and $(f_1,f_2)$ is a pair such that $f_1+f_2 =f$ and $0 \leq f_1 \leq i$ and $0 \leq f_2 \leq g-i$.
Let $S$ be an irreducible component of ${\mathcal M}_g^f$ and let $\bar S$ be its closure in $\bar {\mathcal M}_g$.
\begin{enumerate}
\item 
Then $\bar S$ intersects $\kappa_{i:g-i}(\bar {\mathcal M}_{i;1}^{f_1} \times \bar {\mathcal M}_{g-i;1}^{f_2})$.
\item Each irreducible component of the intersection contains the image of a component of 
$\bar {\mathcal M}_{i;1}^{f_1} \times \bar {\mathcal M}_{g-i;1}^{f_2}$.
\end{enumerate}
\end{proposition}

Let ${\mathcal C} \rightarrow S$ be a relative proper semi-stable 
curve of compact type of genus $g$ over $S$.  Then
${\rm Pic}^0({\mathcal C})[\ell]$ is an \'etale cover of $S$ with geometric fiber
isomorphic to $({\mathbb Z}/\ell)^{2g}$.  For each $n \in {\mathbb N}$, the fundamental group
$\pi_1(S,s)$ acts linearly on the fiber ${\rm Pic}^0({\mathcal C})[\ell^n]_s$, and the
monodromy group ${\sf M}_{\ell^n}({\mathcal C} \rightarrow S, s)$ is the image of $\pi_1(S,s)$
in ${\rm Aut}({\rm Pic}^0({\mathcal C})[\ell^n]_s)$.  
Also ${\sf M}_{{\mathbb Z}_\ell}(S,s):={\rm lim}_{\stackrel{\leftarrow}{n}} {\sf M}_{\ell^n}(S,s)$ is the $\ell$-adic monodromy group. 
When $S$ is an irreducible
component of ${\mathcal M}_g^f$ and ${\mathcal C} \rightarrow S$ is the tautological curve, 
the next result states that
${\sf M}_\ell(S):={\sf M}_\ell({\mathcal C} \to S,s)$ and ${\sf M}_{{\mathbb Z}_\ell}(S):={\sf M}_{{\mathbb Z}_\ell}({\mathcal C} \to S,s)$ are 
as large as possible.

\begin{theorem} \label{Tmono} \cite[Theorem 4.5]{AP:08}
Let $\ell$ be a prime distinct from $p$; let $g \ge 2$ and $0 \leq f \leq g$ with $f \not = 0$ if $g = 2$. 
Let $S$ be an irreducible component of ${\mathcal M}_g^f$, the $p$-rank $f$ stratum in ${\mathcal M}_g$. 
Then ${\sf M}_\ell(S)\simeq {\rm Sp}_{2g}({\mathbb Z}/\ell)$ and
${\sf M}_{{\mathbb Z}_\ell}(S) \simeq {\rm Sp}_{2g}({\mathbb Z}_\ell)$.
\end{theorem}

\subsection{Irreducibility of fibers of $\Pi_\ell$ over ${\mathcal M}_g^f$}

Recall that the morphism $\Pi_\ell: {\mathcal R}_{g, \ell} \to {\mathcal M}_g$, which sends
the point representing the cover $\pi:Y \to X$ to the point representing the curve $X$,
is finite and flat with degree $\ell^{2g}-1$.

\begin{proposition} \label{Pirred}
Under the hypotheses of Theorem \ref{Tmono}, 
if $S$ is an irreducible component of ${\mathcal M}_g^f$,
then $\Pi_\ell^{-1}(S)$ is irreducible.
\end{proposition}

\begin{proof}
Equip $({\mathbb Z}/\ell)^{2g}$ with the standard symplectic pairing $\langle \cdot,\cdot \rangle_{\rm std}$, and let
\begin{equation*}
S_{[\ell]} := {\rm Isom}( ({\rm Pic}^0({\mathcal C}/S)[\ell], \langle\cdot,\cdot\rangle_\lambda), (({\mathbb Z}/\ell)_S^{2g},\langle \cdot,\cdot \rangle_{{\rm std}})).
\end{equation*}
There is an $\ell$th root of unity on $S$, so $S_{[\ell]} \rightarrow S$ is an \'etale Galois cover, possibly
disconnected, with covering group ${\rm Sp}_{2g}({\mathbb Z}/\ell)$.
By Theorem \ref{Tmono}, ${\sf M}_\ell(S) \simeq {\rm Sp_{2g}}({\mathbb Z}/\ell)$.
The geometric interpretation of this is that 
$S_{[\ell]}$ is irreducible.

Suppose $\tilde{\xi}$ is a point of $S_{[\ell]}$.  
Then $\tilde{\xi}$ represents a curve $X$, together with an isomorphism 
between $({\mathbb Z}/\ell)^{2g}$ and ${\rm Pic}^0(X)[\ell]$.  The isomorphism identifies $(1,0, \ldots, 0)$ with a point of order $\ell$ on $J_X$.
It follows that $\tilde{\xi}$ determines an unramified ${\mathbb Z}/\ell$-cover $\pi: Y \to X$.
Thus there is a forgetful morphism $F: S_{[\ell]} \to \Pi_\ell^{-1}(S)$.
Then $\Pi_\ell^{-1}(S)$ is irreducible because $S_{[\ell]}$ is.
\end{proof}

\subsection{Key degeneration result}

\begin{proposition} \label{TdegenW}
Let $g \geq 3$ and $0 \leq f \leq g$.
Let $Q$ be an irreducible component of $\bar W_g^f$.
\begin{enumerate}
\item Then $Q$ intersects $\Delta_{i:g-i}$ for each $1 \le i \le \lfloor g/2 \rfloor$.
\item 
More generally, if $(f_1,f_2)$ is a pair such that $f_1+f_2 =f$ and $0 \leq f_1 \leq i$ and $0 \leq f_2 \leq g-i$, 
then $Q$ contains the image of a component of $\kappa_{i;g-i}(\bar W_{i;1}^{f_1} \times \bar W_{g-i;1}^{f_2})$.
\end{enumerate}
\end{proposition}

\begin{proof}
By Proposition \ref{Pirred}, $Q=\Pi_\ell^{-1}(S)$ for some irreducible component $S$ of $\bar {\mathcal M}_g^f$.
By Proposition \ref{PinterDelta1}, $S$ contains the image of a component of 
$\bar {\mathcal M}_{i;1}^{f_1} \times \bar {\mathcal M}_{g-i;1}^{f_2}$.
Consider a point $\xi$ of $Q$ lying above this image.
Then $\xi$ represents an unramified ${\mathbb Z}/\ell$-cover $\pi:Y \to X$ as in Section \ref{SDeltaig-i}.
By definition, $X$ is a stable curve having components $C_1$ and $C_2$ of genera $i$ and $g-i$ and $p$-ranks $f_1$ and $f_2$.

The ${\mathbb Z}/\ell$-cover $\pi$ is determined by a point of order $\ell$ on $J_X$.
Now $J_X \simeq J_{C_1} \oplus J_{C_2}$, so $J_X[\ell] \simeq J_{C_1}[\ell] \oplus J_{C_2}[\ell]$.
The point $\xi$ is in $\Delta_i$ or $\Delta_{g-i}$ if and only if the point of order $\ell$ is in either $J_{C_1}[\ell] \oplus \{0\}$ or $\{0\} \oplus J_{C_2}[\ell]$.
There are $\ell^{2g}-\ell^{2i}-\ell^{2(g-i)} +1$ points of order $\ell$ which do not have this property.
Since $Q=\Pi_\ell^{-1}(S)$, without loss of generality, one can suppose that the point of order $\ell$ is one of these
or, equivalently, that $\xi$ is in $\Delta_{i:g-i}$, completing part (1).

Every component of $Q \cap \Delta_{i:g-i}$ has dimension $2g-4+f$.  By Lemma \ref{LdimW}(3), this 
equals the dimension of the components of
 $\kappa_{i:g-i}(\bar W_{i;1}^{f_1} \times \bar W_{g-i;1}^{f_2})$, finishing part (2).
\end{proof}

\subsection{Ordinary Pryms}

The first theorem is that the Prym of an unramified ${\mathbb Z}/\ell$-cover of a
generic curve of genus $g$ and $p$-rank $f$ is ordinary, for any $0 \leq f \leq g$.

\begin{theorem} \label{TW}
Let $\ell$ be a prime distinct from $p$; let $g \ge 2$ and $0 \leq f \leq g$ with $f \not = 0$ if $g = 2$. 
If $Q$ is an irreducible component of $W_g^f=\Pi_\ell^{-1}({\mathcal M}_g^f)$, 
then the Prym of the cover represented by the generic point of $Q$ is ordinary (with $p$-rank $f'_Q=(\ell-1)(g-1)$).
\end{theorem}

\begin{proof} The proof is by induction on $g$, with the base case $g=1$ being vacuous.
Suppose the result is true for all $1 \leq g' < g$.
Let $Q$ be an irreducible component of $W_{g}^f$.  
Let $\bar Q$ be its closure in $\bar {\mathcal R}_{g, \ell}$.
Choose $i$ such that $1 \leq i \leq g-1$ and a pair $(f_1,f_2)$ such that $f_1+f_2=f$ and $0 \leq f_1 \leq i$ and $0 \leq f_2 \leq g-i$.
Note that one can avoid the choice $i=2$ and $f_1=0$.
By Proposition \ref{TdegenW}, 
$\bar Q$ contains a component of $\kappa_{i:g-i}(\bar W_{i,1}^{f_1} \times \bar W_{g-i, 1}^{f_2})$.

Let $f'_Q$, $f'_{\partial Q}$, $f'_1$, and $f'_2$ respectively denote
the $p$-rank of the Prym of the cover represented by the generic point of a component of 
$Q$, $\partial Q$, $\bar W_{i;1}^{f_1}$ and $\bar W_{g-i; 1}^{f_2}$.
By semi-continuity $f'_Q \geq f'_{\partial Q}$.
By Lemma \ref{LDelig-i}, $f'_{\partial Q} = f'_1+f'_2 + (\ell-1)$. 
By the inductive hypothesis, $f'_1=(\ell-1)(i-1)$ and $f'_2=(\ell-1)(g-i-1)$.  
Thus $f'_Q \geq (\ell-1)(g-1)$ which equals $\dim(P_\pi)$.
\end{proof}

Theorem \ref{Tintro1}(1) follows from Proposition \ref{Pirred} and Theorem \ref{TW}. 

\section{Purity results} \label{Spurity}

\subsection{A stratification of $\bar{\mathcal R}_{g}$ by the $p$-ranks of $X$ and $P_\pi$}

When $\ell=2$ and $p$ is odd, we consider the stratification of ${\mathcal R}_{g,2}$ by $p$-rank.
Proposition \ref{Phalfway} gives a lower bound for the dimension of the $p$-rank strata.
Since this section is only about double covers, the subscript $\ell=2$ is dropped from the notation for simplicity.

If $\pi: Y \to X$ is an unramified double cover, let $f'$ denote the $p$-rank of $P_\pi$.
Let $\tilde{{\mathcal A}}_{g-1}^{f'}$ denote the $p$-rank $f'$ stratum of 
the toroidal compactification $\tilde{{\mathcal A}}_{g-1}$ of the moduli space ${\mathcal A}_{g-1}$ 
of principally polarized abelian varieties of dimension $g-1$.
The Prym map $Pr_{g}: \bar{\mathcal R}_{g} \rightarrow \tilde{{\mathcal A}}_{g-1}$
sends the point representing $\pi:Y \rightarrow X$ to 
the point representing the principally polarized abelian variety $P_\pi$. 
The image and fibers of $Pr_g$ are well understood only for $2 \leq g \leq 6$.

\begin{definition}
Let $0 \leq f \leq g$ and $0 \leq f' \leq g-1$.  Define $\bar{V}_g^{f'}=Pr_{g}^{-1}(\tilde{{\mathcal A}}_{g-1}^{f'})$
and $V_g^{f'} = \bar{V}_g^{f'} \cap {\mathcal R}_g$. 
Define $\bar{{\mathcal R}}_{g}^{(f,f')}=\bar{W}_g^f \cap \bar{V}_g^{f'}$
and ${\mathcal R}_g^{(f,f')}= \bar{{\mathcal R}}_{g}^{(f,f')} \cap {\mathcal R}_g$. 
\end{definition}

Hence, the points of $V_g^{f'}$ (resp.\ ${\mathcal R}_{g}^{(f,f')}$) represent unramified double
covers $\pi:Y \to X$ of a smooth curve $X$ of genus $g$ (resp.\ and $p$-rank $f$) 
such that $P_\pi$ has $p$-rank $f'$.

By Theorem \ref{TW}, if $p \geq 3$, then ${\mathcal R}_g^{f, g-1}$ is non-empty of dimension $2g-3+f$ 
for all $g \geq 2$ and $0 \leq f \leq g$ with $f \not = 0$ if $g=2$.
Applying purity yields the following result. 

\begin{proposition} \label{Phalfway}
Let $\ell =2$, $p \geq 3$, $g \geq 2$ and $0 \leq f \leq g$.
For $0 \leq f' \leq g-2$, 
if ${\mathcal R}_{g}^{(f,f')}$ (resp.\ $\bar {\mathcal R}_g^{(f,f')}$) is non-empty, 
then each of its components has dimension at least $g-2+(f+f')$.
\end{proposition}

\begin{proof}
Consider the forgetful morphism $\tau_{g}: {\mathcal R}_{g} \to {\mathcal M}_{2g-1}$ which sends the point representing 
$\pi:Y \to X$ to the point representing $Y$.
If $g \geq 2$, then the genus of $Y$ is at least $3$ and ${\rm Aut}_k(Y)$ is finite.
So $\tau_{g}$ is finite-to-$1$ and its image has dimension $3g-3$.

Let $U$ be a component of ${\mathcal R}_g^{(f,f')}$ and let $Z$ be its image under $\tau_g$.
If $\pi:Y \to X$ is represented by a point of $U$, then the $p$-rank of $Y$ is $f+f'$.
Thus $Z$ is contained in ${\mathcal M}_{2g-1}^{f+f'}$.
Note that $Z$ is a component of ${\rm Im}(\tau_{g}) \cap {\mathcal M}_{2g-1}^{f+f'}$.
By Lemma \ref{Lintersect},
\[{\rm codim}({\rm Im}(\tau_{g}) \cap {\mathcal M}_{2g-1}^{f+f'}, {\rm Im}(\tau_g)) \leq {\rm codim}({\mathcal M}_{2g-1}^{f+f'}, {\mathcal M}_{2g-1}).\]

By \cite[Theorem 2.3]{FVdG}, ${\rm codim}({\mathcal M}_{2g-1}^{f+f'}, {\mathcal M}_{2g-1}) = 2g-1-(f+f')$.
Now ${\rm dim}(U)={\rm dim}(Z)$ and so
\[{\rm dim}(U) \geq 3g-3 - (2g-1-(f+f'))=g-2+f+f'.\]

The statement is also true for $\bar {\mathcal R}_g^{(f,f')}$ since ${\mathcal R}_{g}^{(f,f')}$ is open and dense in it.
\end{proof}

\begin{remark}The hypothesis in Proposition \ref{Phalfway} that ${\mathcal R}_{g}^{(f,f')} \neq \emptyset$ is not superfluous.
When $p=3$, then ${\mathcal R}_{2}^{(0,0)} = \emptyset$ \cite[Theorem 6.1]{FVdG}.
\end{remark}

\begin{remark}
The strategy of the proof of Proposition \ref{Phalfway} does not give much information 
for covers of degree $\ell \geq 3$
because $g_Y$ is too big relative to $3g-3$.
\end{remark}

\subsection{Increasing the $p$-rank of the Prym variety}

We show that geometric information about $\mathcal{R}_g^{(f,f')}$ can be used
to deduce geometric information about $\mathcal{R}_g^{(f,F')}$ when $f' \leq F' \leq g-1$. 

\begin{proposition} \label{Pgoup}
Let $g \geq 2$.
If $\mathcal{R}_g^{(f, f')}$ is non-empty and has a component of dimension $g-2 +f +f'$
in characteristic $p$, 
then $\mathcal{R}_g^{(f, F')}$ is non-empty and has a component of dimension $g-2 +f+F'$
in characteristic $p$
for each $F'$ such that $f' \leq F' \leq g-1$.
\end{proposition}

\begin{proof}
Let $S_{f'}$ be a component of $\mathcal{R}_g^{(f, f')}$ having dimension $g-2 +f +f'$.
Then $S_{f'}$ is contained in $W_g^f:=\Pi^{-1}(\mathcal{M}_g^f)$. 
Each component of the latter has dimension $2g-3+f$ since $\Pi:\mathcal{R}_g \to \mathcal{M}_g$ is finite and flat
and $\mathcal{M}_g^f$ is pure of dimension $2g-3+f$ by \cite[Theorem 2.3]{FVdG}.
Thus $S_{f'}$ has codimension $g-1-f'$ in $W_g^f$.
Also, the generic geometric point of $W_g^f$ represents a cover $\pi$ such that the Prym 
$P_\pi$ has $p$-rank $g-1$ by Theorem \ref{TW}.

Consider the forgetful morphism $\tau_{g}: {\mathcal R}_{g} \to {\mathcal M}_{2g-1}$ 
which sends the point representing $\pi:Y \to X$ to the point representing $Y$.
Since $g \geq 3$, the map $\tau_{g}$ is finite-to-$1$. 
Now $\tau_g(S_{f'}) \subset \mathcal{M}_{2g-1}^{f+f'}$ and 
$\tau_g(W_g^f) \subset \mathcal{M}_{2g-1}^{f+g-1}$.
Thus the $p$-ranks and the dimensions for $\tau_g(S_{f'})$ and $\tau_g(W_g^f)$ both differ by
exactly $g-1-f'$.  

By purity, the $p$-rank can only change in codimension $1$.
It follows that there is a nested sequence $T_{f'} \subset \cdots \subset T_i \subset \cdots \subset T_{g-1}$, 
indexed by $i$ from $f'$ to $g-1$,
with $T_{f'}=\tau_g(S_0)$ and $T_{g-1} = \tau_g(W_g^f)$, 
such that ${\rm dim}(T_i)=g - 2 + f + i$ and the generic geometric point 
of $T_i$ represents a curve $Y_i$ with $p$-rank $f+i$.

Then $T_i$ is in the image of $\tau_g$, so there is a sequence 
$S_{f'} \subset \cdots \subset S_i  \cdots \subset W_g^f$ such that $\tau_g(S_i)=T_i$.  
Thus ${\rm dim}(S_i)=g - 2 + f + i$.
Also, the generic geometric point of $S_i$ represents an unramified double cover $\pi:Y_i \to X_i$ 
such that $Y_i$ has $p$-rank $f+i$ and $X_i$ has genus $g$ and $p$-rank $f$; it follows that $P_{\pi_i}$
has $p$-rank $i$.
Thus $R_g^{(f,f')}$ contains an open dense subset of $S_i$, which we denote again by $S_i$
at the risk of causing confusion.

The next claim is that $S_i$ is open and dense in a component of $R_g^{(f,i)}$ 
for $f' \leq i \leq g-1$.
This is true for $i=f'$ by hypothesis.  If it is not true for all $i$, let $j$ be the minimal index for which 
it is false.  Then $S_{j-1}$ has codimension at least 2 inside a component $\Sigma$ of 
$R_g^{(f,j)}$, and $\tau_g(S_{j-1})$ has codimension at least $2$ in $\tau_g(\Sigma)$.
This contradicts purity, since the $p$-rank drops by $1$ on a subset of codimension $2$.
This completes the proof. 
\end{proof}

\section{Results for low genus when $\ell=2$} \label{Slowgenus}

This section contains results about non-ordinary Pryms of unramified double covers of 
curves of low genus $g=2$ and $g=3$.  
Since this section is only about double covers, the subscript $\ell=2$ is dropped from the notation for simplicity.

Recall that ${\mathcal A}_g^f$ is irreducible for all $g \geq 2$ and $0 \leq f \leq g$ except $(g,f)=(2,0)$ \cite[Theorem A]{chaioort}.  When either $g=2$, $f=1,2$ or $g=3$, $0 \leq f \leq 3$, 
the image of ${\mathcal M}_g^f$ under the Torelli map is open and dense in ${\mathcal A}_g^f$ and thus 
${\mathcal M}_g^f$ is irreducible as well.

\subsection{Base Case: Genus 2}\label{BaseCase}

This section contains a proof that ${\mathcal R}_2^{(f,f')}$ is non-empty with the expected dimension for all six choices of $(f,f')$
when $p \geq 5$.

\begin{proposition}\label{prop:basecase}
Let $g=2$, $0 \leq f \leq 2$, and $0 \leq f' \leq 1$.
Then ${\mathcal R}_2^{(f,f')}$ is non-empty (except when $p=3$, $f=0,1$, and $f'=0$)
and each of its components has dimension $f+f'$.
\end{proposition}

\begin{figure}[h]
\begin{tikzpicture}[node distance=2cm,line width=.8pt]

\node(00) at (0,0)     {${\mathcal R}_2^{(0,0)}$};
\node(10)       [right of= 00] {${\mathcal R}_2^{(1,0)}$};
\node(20)      [right of=10]  {${\mathcal R}_2^{(2,0)} $};
\node(01)      [above of=00]       {${\mathcal R}_2^{(0,1)}$};
\node(11)      [right of=01]       {${\mathcal R}_2^{(1,1)}$};
\node(21)      [right of=11]       {${\mathcal R}_2^{(2,1)}$};

\draw(00)      -- (10);
\draw(00)      -- (01);
\draw(01)      -- (11);
\draw(11)      --  (21);
\draw(11)      --  (10);
\draw(21)      --  (20);
\draw(10) -- (20);
\end{tikzpicture}
\caption{The $p$-rank stratification of ${\mathcal R}_2$}

\end{figure}
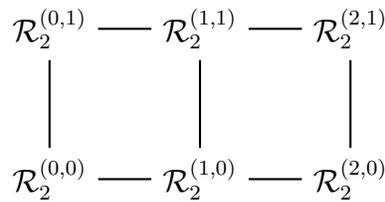

\begin{proof}
By Lemma \ref{LdimW}, $W_2^f$ is non-empty with dimension $1+f$ for $0 \leq f \leq 2$.
If $f=1,2$, then ${\mathcal M}_2^f$ is irreducible and so $W_2^f$ is irreducible by Proposition \ref{Pirred}.

By \cite[Section 7.1]{FVdG}, ${\mathcal R}_2^{(1,0)}$ and ${\mathcal R}_2^{(0,0)}$ are empty when $p=3$.

\begin{itemize}
\item{(0,0).}
By \cite[Theorem 6.1]{FVdG}, if $p \geq 5$, then ${\mathcal R}_2^{(0,0)}$ is nonempty with dimension $0$.

\item{(0,1).}
When $f=0$, then $W_2^0={\mathcal R}_2^{(0,1)} \cup {\mathcal R}_2^{(0,0)}$. 
So ${\mathcal R}_2^{(0,1)}$ is open and dense in $W_2^0$ and thus ${\rm dim}({\mathcal R}_2^{(0,1)})=1$. 

\item{(2,1).} Since $W_2^2$ is open and dense in ${\mathcal R}_2$, which contains ${\mathcal R}_2^{(0,1)}$, 
the generic point of $W_2^2$ has $f'=1$.
This implies that ${\mathcal R}_2^{(2,1)}$ is open and dense in $W_2^2$ and ${\rm dim}({\mathcal R}_2^{(2,1)})=3$.

\item{(1,1).}
Purity, applied to ${\mathcal R}_2^{(0,1)} \subset {\mathcal R}_2$,
shows that ${\mathcal R}_2^{(1,1)}$ is non-empty with dimension $2$.
Thus ${\mathcal R}_2^{(1,1)}$ is open and dense in $W_2^1$. 

\item{(2,0).}
The fiber product construction in Section \ref{Sexample} shows that ${\rm dim}(V_2^0)=2$.
Namely, a point of ${\mathcal A}_1^0$ is represented by a supersingular elliptic curve $E_\lambda$.
By \eqref{Eg2hypcons} in Section \ref{Sexample}, 
there is a 2-dimensional family of curves $X$ with an unramified double cover $\pi:Y \to X$
with $P_\pi \sim E_\lambda$. 

Note that ${\mathcal R}_2^{(2,0)}$ is non-empty and open and dense in $V_2^0$:
if $p=3$, then ${\mathcal R}_2^{(2,0)} =V_2^0$ since ${\mathcal R}_2^{(0,0)}$ and ${\mathcal R}_2^{(1,0)}$ are empty;
if $p \geq 5$, then $W_2^1$ is irreducible with dimension $2$ and generic $f'=1$ and so
no component of $V_2^0$ is contained in $W_2^1$. 

\item{(1,0).}
If $p \geq 5$, applying purity to ${\mathcal R}_2^{(0,0)} \subset V_2^0$ shows that ${\mathcal R}_2^{(1,0)}$ is non-empty.
Also ${\mathcal R}_2^{(1,0)} \subset W_2^1$ which is irreducible of dimension $2$ and generic $f'=1$.  Thus
${\rm dim}({\mathcal R}_2^{(1,0)}) \leq 1$.  
By Proposition \ref{Phalfway}, every component of ${\mathcal R}_2^{(1,0)}$ has dimension $1$.
\end{itemize}
\end{proof}

\begin{remark}
When $g =2$, then 
each of the 15 connected unramified ${\mathbb Z}/2$-covers $\pi:Y \to X$ arises via a fiber product construction.
For a fixed (small) prime $p$, it is thus computationally feasible to find equations for curves represented by
points of ${\mathcal R}_2^{(f,f')}$.  However, if $f+f'$ is small, then it is not feasible to prove that 
${\mathcal R}_2^{(f,f')}$ is non-empty for all primes $p$
using a computational perspective; this is explained in more detail in Section \ref{Sexample}.
\end{remark}

\begin{remark} \label{Rsingular}
Consider ${\mathcal M}_{1;2}^0$, whose points represent supersingular elliptic curves with $2$ marked points. 
Then $\kappa_{0, II}({\mathcal M}_{1;2}^0)$ has dimension $1$ and is fully contained in $\partial {\mathcal R}_2^{(1,0)}$.
\end{remark}

\subsection{Base case: $g=3$}

\begin{proposition}\label{P:NE}
Let $\ell = 2$ and $p \geq 3$.  Let $g=3$ and $0 \leq f \leq 3$ (with $f \not = 0,1$ when $p=3$).  
Then $\Pi^{-1}({\mathcal M}_3^f)$ is irreducible and ${\mathcal R}_3^{(f,1)} = \Pi^{-1}({\mathcal M}_3^f) \cap V_3^{1}$ 
is non-empty with dimension $2+f$.
\end{proposition}

\begin{proof}
For $0 \leq f \leq 3$, ${\mathcal M}_3^f$ is irreducible and so $\Pi^{-1}({\mathcal M}_3^f)$ is irreducible by
Proposition \ref{Pirred}.  Also
${\rm dim}(\Pi^{-1}({\mathcal M}_3^f))=3+f$.
By Theorem \ref{TW}, the Prym of the unramified ${\mathbb Z}/2$-cover represented by the 
generic point of $\Pi^{-1}({\mathcal M}_3^f)$ has $p$-rank $f'=2$.
By Proposition~\ref{Phalfway}, if ${\mathcal R}_3^{(f, 1)}=\Pi^{-1}({\mathcal M}_3^f) \cap V_3^1$ is non-empty, then its components 
have dimension $2+f$.

Recall that ${\mathcal R}_2^{(f,0)}$ is non-empty and has dimension $f$ by Proposition \ref{prop:basecase} when $f=1,2$
and by \cite[Theorem 6.1]{FVdG} (except when $p=3$ and $f=0,1$).

If $0 \leq f \leq 2$, consider $I=\kappa_{2:1}({\mathcal R}_{2;1}^{(f,0)} \times {\mathcal R}_{1;1}^{(0,0)}) \subset \bar {\mathcal R}_3$.
Then $I$ is non-empty.  
The choice of base point increases the dimension by $1$ so ${\rm dim}(I)={\rm dim}(\bar {\mathcal R}_2^{(f,0)}) + 1 = f+1$.
If $f=3$, consider
$I=\kappa_{2:1}({\mathcal R}_{2;1}^{(2,0)} \times {\mathcal R}_{1;1}^{(1,0)}) \subset \bar {\mathcal R}_3$.
Then $I\not = \emptyset$ and ${\rm dim}(I)=4$.

By Lemmas \ref{dimig-i} and \ref{LDelig-i}, $I$ is contained in a component $T$ of $\bar{{\mathcal R}}_3^{(f,1)}$.  
Now ${\rm dim}(T) \geq f+2$ by Proposition \ref{Phalfway}.
The generic point of $T$ is not contained in $\Delta_{2:1}[\bar {\mathcal R}_3^{(f,1)}]$ 
since the dimension of the latter is bounded by $f+1$.  
Moreover, from the construction of $I$, 
the generic point of $T$ is not contained in any other boundary component
and is thus in ${{\mathcal R}}_3^{(f,1)}$.
\end{proof}

\section{Non-ordinary Pryms of unramified double covers} \label{Snotord}


In this section, we demonstrate the existence of smooth curves of given genus and $p$-rank
having an unramified double cover whose Prym is not ordinary.
Since this section is only about double covers, the subscript $\ell=2$ is dropped from the notation for simplicity.

\subsection{Almost ordinary}

Theorem \ref{Tintro2}(1) follows from Theorem \ref{T:NE}, 
which states that there is a codimension one condition on a generic curve $X$ of genus $g$ and $p$-rank $f$
for which the Prym of an unramified double cover $\pi:Y \to X$ is almost ordinary.
The almost ordinary condition means that the $p$-rank of $P_\pi$ is $g-2={\rm dim}(P_\pi) -1$. 

Consider the stratum
${\mathcal R}_g^{(f,g-2)}=W_g^f \cap V_g^{g-2}$ whose points represent unramified double covers $\pi:Y \to X$ 
such that $X$ is a smooth curve of genus $g$ and $p$-rank $f$ and such that $P_\pi$ has $p$-rank $f'=g-2$
(or, equivalently, such that $P_\pi$ is almost ordinary).

\begin{theorem}\label{T:NE}
Let $\ell=2$ and $p \geq 3$. Let $g \geq 2$ and $0 \leq f \leq g$ (with $f \geq 2$ when $p=3$).  Let $f'=g-2$.
Then ${\mathcal R}_g^{(f, g-2)}$ is non-empty and each of its components has dimension $2g-4+f$.

More generally, let $S$ be a component of ${\mathcal M}_g^f$.  
Then the locus of points of $\Pi^{-1}(S)$ representing unramified double covers for which the Prym $P_\pi$ is almost ordinary 
is non-empty and codimension $1$ in $\Pi^{-1}(S)$ (dimension $2g-4+f$). 
\end{theorem}

\begin{proof}
The first statement follows from the second since every component of ${\mathcal R}_g^{(f,g-2)}$ is contained in 
$\Pi^{-1}(S)$ for some component $S$ of ${\mathcal M}_g^f$.
It thus suffices to prove
that
$\Pi^{-1}(S) \cap V_g^{g-2}$ is non-empty and its components have dimension $2g-4+f$.

\bigskip

{\bf Dimension:}
Let $T$ be a component of $\Pi^{-1}(S) \cap V_g^{g-2}$.  Then Proposition~\ref{Phalfway} implies that ${\rm dim}(T) \geq 2g-4+f$. 
The generic point of $\Pi^{-1}(S)$ represents a cover whose Prym has $p$-rank $f'=g-1$ by Theorem \ref{TW}
(or Proposition \ref{prop:basecase} if $g=2$ and $f=0$). 
Thus ${\rm dim}(T)={\rm dim}(\Pi^{-1}(S))-1=2g-4+f$.
It thus suffices to show $\Pi^{-1}(S) \cap V_g^{g-2}$ is non-empty.

\bigskip

{\bf Base cases:}
When $g=2$ and $0 \leq f \leq 2$, then $\Pi^{-1}(S) \cap V_2^0$ is non-empty with dimension $f$ by Proposition \ref{prop:basecase} (unless $p=3$ and $f=0,1$).
When $g =3$ and $0 \leq f \leq 3$, then $\Pi^{-1}({\mathcal M}_3^f)$ is irreducible and 
$\Pi^{-1}({\mathcal M}_3^f) \cap V_g^{1}$ is non-empty with dimension $2+f$ by Proposition \ref{P:NE}.

\bigskip

{\bf Strategy:}
Suppose $g \geq 4$. Let $\bar S$ be the closure of $S$ in $\bar {\mathcal M}_g^f$.
The plan is to show that $\Pi^{-1}(\bar S) \cap \bar V_g^{g-2}$ is non-empty and 
that one of its components is not contained in $\partial {\mathcal R}_g$.

\bigskip

{\bf Non-empty:}
Let $g_1=3$ and $g_2=g-3$.
Choose $f_1, f_2$ such that $f_1+f_2=f$ with $0 \leq f_i \leq g_i$.
By Proposition \ref{PinterDelta1}, 
there are components $S_{i;1}$ of $\bar {\mathcal M}_{g_i;1}^{f_i}$ such that
\[\kappa_{g_1:g_2}(S_{1;1} \times S_{2;1}) \subset \bar{S}.\]
Recall the forgetful map $\psi_M: {\mathcal M}_{g;1} \to {\mathcal M}_g$ from Section \ref{Smarkedcurves}.
Let $S_i = \psi_M(S_{i;1})$ which is an
irreducible component of $\bar {\mathcal M}_{g_i}^{f_i}$.

The Prym of the cover represented by the generic point of $\Pi^{-1}(S_{2})$ has $p$-rank $f'_2=g_2-1$ by Theorem \ref{TW}.
By Proposition \ref{Pirred}, $\Pi^{-1}(S_1)$ is irreducible.
By Proposition \ref{P:NE}, there exists a point of $\Pi^{-1}(S_{1})$ representing a cover whose Prym has $p$-rank $f'_1=g_1-2=1$.
Since $\bar {\mathcal M}_3^{f_1}$ is irreducible, $\Pi^{-1}(S_{1}) = \bar W_{3}^{f_1}$.
Consider
\[N:=\kappa_{g_1:g_2}(\Pi^{-1}(S_{1;1}) \times \Pi^{-1}(S_{2;1})) \subset \Pi^{-1}(\bar{S}).\]
By Lemma \ref{LDelig-i}, $N$ contains a point representing a cover whose Prym has $p$-rank $f'=f_1'+f_2'+1=g-2$,
i.e., whose Prym is almost ordinary. 
Thus $\Pi^{-1}(\bar S) \cap V_g^{g-2}$ is non-empty.

\bigskip

{\bf Generically smooth:}
Let $T$ be a component of $\Pi^{-1}(\bar S) \cap \bar V_g^{g-2}$ containing $N$.
By the remarks above, $T$ intersects the image of
\[\kappa_{g_1:g_2}: \bar {\mathcal R}_{g_1;1}^{(f_1,1)} \times \bar W_{g_2;1}^{f_2} \rightarrow \Delta_{3:g-3}[\bar {\mathcal R}_{g}^{(f,g-2)}].\] 
This image has dimension
\[(2+f_1) + 1 + (2(g-3)-3+f_2)+1=2g - 5 +f.\]
By Proposition \ref{Phalfway}, ${\rm dim}(T) \geq 2g-4 +f$.
Thus the generic point $\tau$ of $T$ is not contained in $\Delta_{3:g-3}[\bar {\mathcal R}_g]$.
Furthermore, $\tau$ is not contained in any other component of $\partial {\mathcal R}_g$
because the generic points of $S_1$ and $S_2$ represent smooth curves.
Thus $\tau \subset \Pi^{-1}(S) \cap V_g^{g-2}$.
\end{proof}

\smallskip

\subsection{Pryms with low $p$-rank}

As an application, we demonstrate the existence of smooth curves of given genus and $p$-rank having an 
unramified double cover whose Prym has any $p$-rank between $\frac{g}{2}-1$ and $g-3$.

\begin{theorem} \label{Tbigg}
Let $p \geq 5$.
Let $g \geq 2$ and write $g=3r+2s$ for integers $r,s \geq 0$.
Let $0 \leq f \leq g$.
Let $2r+s-1 \leq f' \leq g-1$.
Then $\mathcal{R}_g^{(f, f')}$ is non-empty and has a component of dimension $g-2 +f +f'$
in characteristic $p$.
\end{theorem}

\begin{proof}
In light of Proposition \ref{Pgoup}, it suffices to prove the result when $f' =2r+s-1$.
The proof is by induction on $r+s$.
In the base case $(r,s)=(0,1)$, then $g=2$ and the result is true by Proposition \ref{prop:basecase}.
In the base case $(r,s)=(1,0)$, then $g=3$ and the result is true by Proposition \ref{P:NE}.
As an inductive hypothesis, suppose that the result is true for all pairs $(r',s')$ such that $1 \leq r'+s' < r+s$.

Case 1: suppose $r \geq 1$.
Let $g_1=3$ and $g_2=g-3$.
There exist $f_1, f_2$ such that $f_1+f_2=f$ and $0 \leq f_1 \leq g_1$ and $0 \leq f_2 \leq g_2$.
Let $f_1'=1$ and $f_2'=2r+s-3$.
By Proposition~\ref{P:NE}, $\mathcal{R}_3^{(f_1,1)}$ is non-empty and has a component $S_1$ of dimension $d_1= 2+f_1$.
(The points of $S_1$ represent unramified double covers $\pi_1:Y_1 \to X_1$ of a smooth curve of genus $g_1$ 
and $p$-rank $f_1$, such that $P_{\pi_1}$ has $p$-rank $1$.)

By the inductive hypothesis applied to $(r-1, s)$, it follows that $\mathcal{R}_{g_2}^{(f_2, 2r+s-3)}$ is non-empty and 
has a component $S_2$ of dimension $d_2= g_2-2 +f_2 + f_2'$.
(The points of $S_2$ represent unramified double covers $\pi_2:Y_2 \to X_2$ of a smooth curve of genus $g_2$ 
and $p$-rank $f_2$, such that $P_{\pi_2}$ has $p$-rank $f_2'$.) 
Adding a marking increases the dimension by $1$, 
so $U_1:= \psi_R^*(S_1) = S_1 \times_{{\mathcal R}_{3}} {\mathcal R}_{3;1}$ has dimension 
$d_1+1= 3+f_1$
and
$U_2:=\psi_R^*(S_2) = S_2 \times_{{\mathcal R}_{g_2}} {\mathcal R}_{g_2;1}$ has dimension 
$d_2+1 = g_2-1+f_2+f_2'$.

Let $\mathcal{K}$ be a component of $\kappa_{3,g_2}(U_1 \times U_2)$; then $\mathcal{K}$ has dimension $d_1+d_2+2$.
By Lemmas \ref{dimig-i} and \ref{LDelig-i}, $\mathcal{K}$ is contained in a component $\mathcal{Z}$ of 
$\bar{\mathcal{R}}_{g_1+g_2}^{(f_1+f_2,f'_1+f'_2+1)}= \bar{\mathcal{R}}_{g}^{(f,2r+s-1)}$.
In other words, the points of $\mathcal{K}$ represent unramified double covers of curves (of compact type)
having genus $g$ and $p$-rank $f$ whose Prym varieties have $p$-rank $f'$.
By Lemma \ref{Lintersect}, the dimension of $\mathcal{Z}$ is at most
\[d_1+d_2 + 3= (g_2+3) - 2 + (f_1+f_2) + (2r+s-1) = g-2+f +f'.\]
Also ${\rm dim}(\mathcal{Z}) \geq g-2+f +f'$ by purity.
Thus ${\rm dim}(\mathcal{Z}) = g-2+f +f'$ and the generic point of $\mathcal{Z}$ is not contained in $\mathcal{K}$.
The generic geometric points of $S_1$ and $S_2$ represent unramified double covers of smooth curves by hypothesis.
Thus the generic geometric point of $\mathcal{Z}$ is not contained in any other boundary component of 
$\bar{\mathcal{R}}_g$
and so it represents an unramified double cover of a smooth curve.

Case 2: suppose $s \geq 1$.  Let $g_1=2$ and $g_2=g-2$.
There exist $f_1, f_2$ such that $f_1+f_2=f$ and $0 \leq f_1 \leq g_1$ and $0 \leq f_2 \leq g_2$.
Let $f_1'=0$ and $f_2'=2r+s-2$.
By Proposition \ref{prop:basecase}, $\mathcal{R}_2^{(f_1,0)}$ is non-empty and has a component $S_1$ 
of dimension $d_1= f_1$.
By the inductive hypothesis applied to $(r, s-1)$, it follows that $\mathcal{R}_{g_2}^{(f_2, 2r+s-2)}$ is non-empty and 
has a component $S_2$ of dimension $d_2= g_2-2 +f_2 + (2r+s -2)$.
The rest of the proof follows the same reasoning as in Case (1).

\end{proof}

\begin{corollary} \label{Cdownmore}
Let $\ell=2$ and $p \geq 5$.  Let $g \geq 4$ and $0 \leq f \leq g$. Suppose $\frac{g}{2}-1 \leq f' \leq g-3$.
Then ${\mathcal R}_g^{(f,f')}$ is non-empty and has a component of dimension $g-2+f+f'$.
In particular, there exists a smooth curve $X/\bar{{\mathbb F}}_p$ of genus $g$ and $p$-rank $f$
having an unramified double cover $\pi:Y \to X$ for which the Prym $P_\pi$ has $p$-rank $f'$.
\end{corollary}

\begin{proof}
If $g$ is even, let $r=0$ and $s=g/2$.  If $g$ is odd, let $r=1$ and $s=(g-3)/2$.
In either case, the condition $\frac{g}{2}-1 \leq f' \leq g-3$ implies that the hypothesis $2r+s-1 \leq f' \leq g-1$ in Theorem \ref{Tbigg}
is satisfied and the result follows from Theorem \ref{Tbigg}.
\end{proof}

\section{Applications to Theta divisors} \label{Stheta}

\subsection{Background}
Recall the definition of the theta divisor from \cite[Section 1.1]{raynaud02}.
Given a relative curve $X/S$, 
let $X^1$ be the curve induced by base change by the absolute Frobenius of $S$.
Consider the relative Frobenius morphism $F:X \to X^1$.
The {\it sheaf of locally exact differentials} $B$ is the image of
$F_*d: F_* ({\mathcal O} _X) \to F_* (\Omega^1_X)$.
There is an exact sequence of ${\mathcal O}_{X^1}$-modules:
\[0 \to {\mathcal O}_{X^1} \to F_*({\mathcal O} _X) \stackrel{F_* d}{\to} B \to  0.\]
Also, $B$ is the kernel of the Cartier operator $C: F_*(\Omega^1_X) \to \Omega^1_{X^1}$, 
and there is an exact sequence of ${\mathcal O}_{X^1}$-modules:
\[0 \to B \to F_*(\Omega^1_X) \stackrel{C}{\to} \Omega^1_{X^1} \to 0.\]
Now $B$ is a vector bundle on $X^1$ of rank $p-1$ and slope $g-1$, 
where the slope is the quotient of the degree by the rank.
More precisely, if $X^1$ is not smooth, then $B$ is a torsion-free sheaf, which is locally free of rank $p-1$ outside the singularities of $X^1$.

By \cite[Theorem 4.1.1]{raynaudsections}, $B$ admits a theta divisor $\Theta_X$. 
This is a positive Cartier divisor on the Jacobian $J^1$ of $X^1$
(the determinant of the universal cohomology).
A point $a \in J^1(k)$ is in the support of $\Theta_X$ 
if and only if $H^0(X^1, B \otimes L_a) \not = 0$ where 
$L_a \in {\rm Pic}^0(X^1)$ is the invertible sheaf identified with $a$.

\subsection{The theta divisor} 
\label{Sprevious}

By work of Raynaud, the theta divisor $\theta_X$ determines whether
unramified covers of the curve $X$ are ordinary.
By \cite[Proposition 1]{raynaudrevetements}, $X$ is ordinary if and only if 
$\theta_X$ does not contain the identity of $J^1$.

To generalize this, consider a non-trivial point $a \in J^1[\ell]$ with $p \nmid \ell$.
The point $a$ determines an unramified ${\mathbb Z}/\ell$-cover $\pi_a: Y_a \to X$, 
and an invertible sheaf $L_a \in {\rm Pic}^0(X^1)$ of order $\ell$.
Denote the orbit of $a$ under $({\mathbb Z}/\ell)^*$ as
${\rm Sat}(a) =\{ia \mid {\rm gcd}(i,\ell)=1\}$.

\begin{proposition} \label{P1=2} \cite[Proposition 2.1.4]{raynaud02}
Let $a \in J^1[\ell]$ be non-trivial with $p \nmid \ell$.
The new part of $\pi_a:Y_a \to X$ is ordinary if and only if 
${\rm Sat}(a)$ does not intersect the theta divisor $\Theta_X$.
\end{proposition}




Using the geometry of $\Theta_X$, Raynaud and Pop/Saidi prove:

\begin{theorem} \label{Traynaud}
Let $X$ be a smooth projective $k$-curve of genus $g \geq 2$.
\begin{enumerate}
\item \cite[Theorem 4.3.1]{raynaudsections}
For sufficiently large $\ell$, there is an unramified ${\mathbb Z}/\ell$-cover $\pi:Y \to X$ such that $P_\pi$ is ordinary.
It suffices to take $\ell > (p-1)3^{g-1}g!$ by \cite[Remark 3.1.1]{tamagawa}.

\item \cite[Theorem 2]{raynaudrevetements}
There is an unramified Galois cover $Z \to X$, with solvable prime-to-$p$ Galois group, 
with a non-ordinary representation (so $Z$ is not ordinary).
\item \cite[Proposition 2.3]{popsaidi}
If $X$ is non-ordinary or if $J_X$ is simple then
there is an unramified ${\mathbb Z}/\ell$-cover $\pi_\ell:Y_\ell \to X$ such that $P_{\pi_\ell}$ is not ordinary
for infinitely many primes $\ell$.
\end{enumerate}
\end{theorem}

\subsection{Comparison with previous work}

The results in this paper strengthen the results in Theorem \ref{Traynaud} 
for a generic curve $X$ of genus $g$ and $p$-rank $f$ for all $g \geq 2$ and $0 \leq f \leq g$.
Specifically,
Theorem \ref{Tintro1} removes
the condition on $\ell$ in Theorem \ref{Traynaud}(1) and shows that all (not just one) of the Pryms 
of the ${\mathbb Z}/\ell$-covers of $X$ are ordinary, for a generic curve $X$ of genus $g$ and $p$-rank $f$.
Theorem \ref{Tintro2} and Corollary \ref{Cdownmore} are about double covers, 
rather than covers of unknown degree, and 
they determine the value of the $p$-rank of the Prym
which gives more information than saying that the new part of the Prym is not ordinary.

\subsection{New results on theta divisors}

We apply Proposition \ref{P1=2} in the opposite direction from Raynaud and Pop/Saidi
to complete the proofs of Theorems \ref{Tintro1}(2) and \ref{Tintro2}(2).

\begin{theorem} \label{Ttheta1}
Let $\ell \not = p$ be prime.
Let $g \geq 2$ and $0 \leq f \leq g$ with $f \not = 0$ if $g=2$.
Let $S$ be an irreducible component of ${\mathcal M}_g^f$.
If $X$ is the curve represented by the generic point of $S$, 
then the theta divisor $\Theta_X$ of the Jacobian of $X$ does not contain any point of order $\ell$.
\end{theorem}

\begin{proof}
By Proposition \ref{P1=2}, 
this statement is equivalent to Theorem \ref{Tintro1}(1) since the Prym is the new part of $\pi:Y \to X$. 
\end{proof}

\begin{theorem} \label{Ttheta2}
Let $\ell =2$.  Let $g \geq 2$ and $0 \leq f \leq g$ (with $f \geq 2$ when $p =3$).
Let $S$ be an irreducible component of ${\mathcal M}_g^f$.
The locus of points of $S$ representing curves $X$ for which $\Theta_X$
contains a point of order $2$ is non-empty with codimension $1$ in $S$.
\end{theorem}

\begin{proof} 
This follows from Theorem \ref{Tintro2}(1) by Proposition \ref{P1=2}. 
\end{proof}

\begin{remark} \label{RnoLtorsion}
Let $g$ be odd and $d=(g-1)/2$.  If $X$ has genus $g$, then $J_X$ contains
the top difference variety $V_d = X_d - X_d$, 
which consists of divisors of the form $\sum_{i=1}^d P_i - \sum_{i=1}^d Q_i$.
By \cite[Corollary~0.4]{F_generalL}, $V_d$ contains no points of order $\ell$ when $X$ is generic.
\end{remark}

\section{Examples and open questions} \label{Sopen}

This section contains examples for $g=2$, a question about Pryms of hyperelliptic curves,
and questions about non-ordinary Pryms whose answers would generalize Theorem \ref{Tintro2}(1).

\subsection{The fiber product construction when $g=2$} \label{Sexample}

We explain why the fiber product construction of unramified double covers
is not useful for proving Proposition \ref{prop:basecase}. 

Suppose $X$ is a genus $2$ curve and $f_1: X \to {\mathbb P}^1$ is a 
hyperelliptic cover branched above a set $B_X$ of cardinality $6$.  
For a set $B_E \subset B_X$ of cardinality $4$, let $f_2:E \to {\mathbb P}^1$ be the hyperelliptic cover
branched above $B_E$.  
The fiber product $f:Y \to {\mathbb P}^1$ of $f_1$ and $f_2$ is a Klein four cover of ${\mathbb P}^1$.
By Abhyankar's Lemma, the degree two subcover $\pi:Y \to X$ is unramified since $B_E \subset B_X$.
Then $J_Y \sim J_X \oplus E$ by \cite[Theorem B]{kanirosen}.  Thus 
$P_\pi \sim E$.

Furthermore, each of the 15 connected unramified double covers $\pi:Y \to X$ arises via the fiber product construction (from one of the 15 
choices of $B_E \subset B_X$).  This is because the hyperelliptic involution $\iota$ on $X$
fixes each point of order $2$ on $J_X$ and thus extends to $Y$. 

For $\lambda \in k - \{0,1\}$, let $E_\lambda: y_2^2=x(x-1)(x-\lambda)$. 
For distinct $t_1, t_2 \in k - \{0,1,\lambda\}$, consider the genus two curve
\begin{equation} \label{Eg2hypcons}
X:y_1^2=f_\lambda(t_1, t_2) := x(x-1)(x-\lambda)(x-t_1)(x-t_2).
\end{equation}
As above, $E_\lambda \sim P_\pi$ for an unramified double cover $\pi:Y \to X$. 
One says that $\lambda$ is supersingular when $E_\lambda$ is supersingular.

Let $M_\lambda(t_1,t_2)$ be the matrix of the Cartier operator on $H^0(X, \Omega^1)$
with respect to the basis $\{dx/y, xdx/y\}$.  Let $c_i$ be the coefficient of $x^i$ in 
$f_\lambda(t_1,t_2)^{(p-1)/2}$.
By \cite[page 381]{Yui}, 
\[M_\lambda(t_1,t_2)=\left (\begin{array}{cc}
c_{p-1} & c_{p-2} \\
c_{2p-1} & c_{2p-2}
\end{array}
\right ).\]

Let $D_\lambda={\rm det}(M_\lambda(t_1,t_2))$
and let $S_\lambda \subset {\mathbb A}^2$ be the vanishing locus of $D_\lambda$.
By \cite[Theorem 2.2]{Yui}, $X$ is ordinary if and only if $D_\lambda \not = 0$; the $p$-rank of $X$ is 
the rank of $N_\lambda(t_1,t_2)=M_\lambda(t_1,t_2)^{(p)} M_\lambda(t_1,t_2)$ 
(where $(p)$ means to raise each entry of the matrix to the $p$th power).

\begin{enumerate}
\item The case $(f,f')=(2,0)$.
For each supersingular $\lambda$, to show $X$ is generically ordinary, one 
needs to check that
$D_\lambda \in k[t_1,t_2]$ is non-zero.

\item The case $(f,f')=(\leq 1, 0)$.
To show ${\mathcal R}_2^{(1,0)} \cup {\mathcal R}_2^{(0,0)} \not = \emptyset$, one needs to find $\lambda$ supersingular such that 
$D_\lambda \in k[t_1,t_2]$ is non-constant and $S_\lambda$ is not contained in the union $L$ of the lines 
$t_i=0$, $t_i=1$, $t_i=\lambda$, and $t_1=t_2$.

\item The case $(f,f')=(0,0)$ for $p \geq 5$.
To show ${\rm dim}({\mathcal R}_2^{(0,0)})=0$, one needs to show that $N_\lambda(t_1,t_2)$
has rank $1$ (not $0$) for every supersingular $\lambda$ and for each generic point of $S_\lambda$ not in $L$.
To show ${\mathcal R}_2^{(0,0)} \not = \emptyset$, one needs to find $\lambda$ supersingular 
and distinct $t_1, t_2 \in k-\{0,1, \lambda\}$ such that $N_\lambda(t_1,t_2)$ has rank $0$.
\end{enumerate}

\begin{example}
(Example of Proposition \ref{prop:basecase}) Let $p=5$.
%
%
  %
%
%
Let $\lambda=a^4 $ for a root $a$ of $x^2 + 4x + 2$. Then $E_\lambda$ is supersingular and 
\[D_\lambda=(t_1 + 4t_2)^2(t_1^2t_2 + t_1t_2^2 + a^{17}t_1^2  + a^{17}t_2^2 + a^5t_1t_2 + a^4t_1  + a^4t_2).\]

\begin{enumerate}
\item $(f,f')=(2,0)$.  Since $D_\lambda \neq 0$, $X$ is generically ordinary.
\item $(f,f')=(1, 0)$. 
When $(t_1,t_2)=(a^{16}, a)$, then $f_X=1$. 
\item $(f,f')=(0,0)$.  By \cite[Section 7.2]{FVdG}, there is exactly one unramified 
double cover
$\pi: Y \to X$ up to isomorphism such that $X$ has genus $2$ and $Y$ has $p$-rank $0$.  An equation for $X$ is 
$y^2=x(x^4+x^3+2x+3)$.
\end{enumerate}
\end{example}

\subsection{The hyperelliptic case} \label{Shyperelliptic}

We expect there is an analogue of Theorem \ref{TW} for the hyperelliptic locus ${\mathcal H}_g$.
One can ask if the Prym of the cover represented by the 
generic point of each irreducible component of $\Pi_\ell^{-1}({\mathcal H}^f_g)$
is ordinary for $1 \leq f \leq g$.
Propositions~\ref{PinterDelta1} and \ref{TdegenW} are true (for $i=1$) for ${\mathcal H}_g^f$ 
\cite[Corollary 3.13]{APprh}
and Theorem~\ref{Tmono} is true for ${\mathcal H}_g^f$ when $f >0$ 
(or for $f=0$ and $\ell >> 0$) \cite[Theorems 5.2, 5.7]{APprh}.
However, there may be complications with Propositions \ref{Pdivisor1}, \ref{Pdivisor2} for ${\mathcal H}_g^f$,
especially when $\ell=2$.  

\subsection{A question for $g=3$ about Pryms of $p$-rank $0$}

\begin{question} \label{Qg=3(0,0)}
For a prime $p$, is ${\mathcal R}_3^{(0,0)}$ non-empty?
Does there exist an unramified double cover $\pi:Y \to X$ of a smooth curve $X/\bar{{\mathbb F}}_p$ of 
genus $3$ such that $Y$ has $p$-rank $0$?
\end{question}

The answer to Question \ref{Qg=3(0,0)} is yes when $p=3$ by \cite[Example~5.5]{FVdG} but is unknown for $p \geq 5$.
By \cite[Proposition 4.2]{FVdG}, if it is non-empty, then ${\rm dim}({\mathcal R}_3^{(0,0)})=1$; 
however, there are components of
$\partial \bar{R}_3^{(0,0)}$ which have dimension $1$ or $2$.

\subsection{Non-ordinary Pryms for odd degree cyclic covers}

It is unknown whether Theorem \ref{T:NE} can be generalized to the case $\ell \geq 3$, for a given prime $p$.

\begin{question} \label{Q2}
Suppose $\ell \not = p$ is an odd prime.  For which $(g,f)$ does there exist a curve $X$ of genus $g$ and $p$-rank $f$
with an unramified ${\mathbb Z}/\ell$-cover $\pi:Y \to X$ such that $P_\pi$ is non-ordinary?
\end{question}

\begin{example} \cite[Section 6]{nakajima}
Let $p=2$ and $g=2$ and $\ell=3$.  
If $X$ is a curve of genus $2$ which is not ordinary ($f < 2$), then the Prym of every unramified ${\mathbb Z}/3$-cover of $X$ is ordinary.
\end{example}

\subsection{A question about purity}

Let $\ell=2$ and $p \geq 3$.
The points of ${\mathcal R}_g^{(f,f')}=W_g^f \cap V_g^{f'}$ represent unramified double covers $\pi:Y \to X$ 
such that $X$ is a smooth curve of genus $g$ and $p$-rank $f$ and $P_\pi$ has $p$-rank $f'$.
By Proposition \ref{Phalfway},
${\rm dim}({\mathcal R}_{g;2}^{(f,f')}) \geq g-2 + f+f'$.

\begin{question} \label{MainConj}
Let $g \geq 2$ and $0 \leq f \leq g$ and $0 \leq f' \leq g-1$.
If ${\mathcal R}_{g; 2}^{(f,f')}$ is non-empty, do all its components have dimension exactly $g-2+f+f'$?
\end{question}

The answer to Question \ref{MainConj} is yes for any $0 \leq f \leq g$ when:
\begin{enumerate}
\item $f'=g-1$ by Theorem \ref{TW} (or Proposition \ref{prop:basecase});

\item and $f'=g-2$ (with $f \geq 2$ when $p=3$) by Theorem \ref{T:NE}.
\end{enumerate}

One complication in answering Question \ref{MainConj} for $f' < g-2$ is that there are families of {\it singular} curves in
$\bar{{\mathcal R}}_{g; 2}^{(f,f')}$ whose dimension exceeds $g-2+f+f'$ as in Remark \ref{Rsingular}.

\subsection{Pryms with $p$-rank zero}

Let $\ell=2$ and $p \geq 5$ and $g=3$.  
Consider $\bar{V}_3^1=Pr_3^{-1}(\tilde{{\mathcal A}}_2^1)$ whose points represent unramified double
covers $\pi:Y \to X$, where $X$ has genus $3$ and $P_\pi$ has $p$-rank $1$.
By \cite[Theorem 4.2, equations 3.14-3.16]{Verra}, $V_3^1$ has one component of dimension 5, 
and three exceptional components of lower dimension.

In $\bar{V}_3^1$ is the locus $V_3^0$ (additional constraint that $f'=0$).
By \cite[Theorem 4.2]{Verra}, ${\rm dim}(V_3^0)=3 + {\rm dim}({\mathcal A}_2^0)=4$.  
If $p > 11$, then ${\mathcal A}_2^0$ is not irreducible by \cite[Theorem 5.8]{katsuraoort}, 
and so $V_3^0$ is not irreducible.
Also in $\bar{V}_3^1$ is the locus $W_3^2 \cap \bar{V}_3^{1}$, 
(additional constraint that $f=2$).
By Proposition \ref{P:NE}, ${\rm dim}(W_3^2 \cap \bar{V}_3^{1})=4$;
it is not known whether it is irreducible.

\begin{question}
Is ${\rm dim}(W_3^2 \cap V_3^0) = 3$?
\end{question}

\bibliographystyle{plain}
\bibliography{prym.bib}

\end{document}